\documentclass[12pt]{article}
\usepackage{amsmath,amsthm,amssymb}
\usepackage[english]{babel} 
\usepackage{array}
\usepackage{rotating}
\usepackage{comment}
\usepackage{graphicx}
\usepackage{pdflscape}
\usepackage{url}
\usepackage[latin1]{inputenc}
\usepackage{tikz}
\newcommand\indicator{{\mathbb I}}
\newcommand\normal{{\cal N}}
\newcommand\Poi{{\rm Poi}}

\newcommand\convD{{\buildrel {\cal D} \over \longrightarrow}}
\newcommand\inL{\buildrel {{\cal L}_1} \over  \longrightarrow}
\newcommand\OL{O_{{\cal L}_1}}
\newcommand\given{\, \vert \, }
\newcommand\E{{\mathbb E}}

\newcommand\V{{\mathbb V{\rm ar}}}
\newcommand\Cov{{\mathbb C{\rm ov}}}
\newtheorem{theorem}{Theorem}[section]
\newtheorem{corollary}{Corollary}[section]
\newtheorem{lemma}{Lemma}[section]
\newtheorem{proposition}{Proposition}[section]

\newtheorem{remark}{Remark}[section]

\newcommand\inprob{\buildrel {\cal P} \over  \longrightarrow}

\newcommand\Prob{{\mathbb P}}
\newcommand\field{{\mathbb F}^{(\theta)}}
\begin{document}  
\begin{center}
{\huge \bf The containment profile of hyperrecursive trees}

\bigskip
\bigskip
{\large \bf Joshua Sparks\footnote{Department of Statistics,
The George Washington University,
Washington, D.C. 20052, U.S.A. }\qquad  Srinivasan Balaji\footnote{Department of Statistics,
The George Washington University,
Washington, D.C. 20052, U.S.A. }

\medskip
  Hosam Mahmoud\footnote{Department of Statistics,
The George Washington University,
Washington, D.C. 20052, U.S.A. }}

\medskip
\bigskip
\today
\end{center}

\bigskip\noindent
{\bf Abstract}
We investigate vertex levels of containment in a random hypergraph grown in the spirit of a recursive tree. We consider a local profile tracking the evolution of the 
containment of a particular vertex over time, and a global profile concerned about counts of the number of vertices of a particular containment level.

For the local containment profile, we obtain the exact
mean, variance and probability distribution 
in terms of standard combinatorial quantities like generalized harmonic numbers and Stirling numbers of the first kind. 
Asymptotically, we observe phases: the early vertices have an asymptotically normal distribution, intermediate vertices have a Poisson distribution, and late vertices have a degenerate distribution.

As for the global containment profile, we establish an asymptotically normal  distribution for the number of vertices at the  smallest containment  
level as
well as their covariances with  
the number of vertices at the  second smallest containment level and the variances of these numbers. The orders in the variance-covariance
matrix establish concentration laws. 

\bigskip
\noindent{\bf AMS subject classifications:} Primary:
    05082, 
   90B15; 
   Secondary: 60C05,     
      60F05.     

\medskip\noindent
{\bf Keywords:} Hypergraph, limit law, phase transition, martingale.

\bigskip
\section{Hyperrecursive trees}
Hypergraphs are generalizations of graphs. In a hypergraph, we have vertices and hyperedges consisting of collections of vertices. Also, the recursive tree is a well-studied structure; see~\cite{Drmota,Hofri,Karonski, Smythe}, among many other sources.
We propose in this paper a generalization of recursive trees to become 
hypergraphs.

A  hyperrecursive tree with parameter (hyperedge size) $\theta$ grows as follows. 
Initially, there are $\theta$ originating vertices, 
all labeled with 0.  These $\theta$ vertices constitute the first hyperedge. 
At each subsequent step, a vertex is added to the structure. The incoming 
vertex chooses 
$\theta-1$ existing vertices to co-share a hyperedge. 
A vertex joining at time~$n$ is labeled with $n$.
The choice of the vertices for the new hyperedge is done uniformly at random,
with all subsets of vertices of size $\theta-1$ being equally likely. The usual recursive tree is one with the parameter
$\theta = 2$.

Figure~\ref{Fig:hyper} illustrates the growth of a hyperrecursive tree in two steps (i.e., at time $n=2$). In this example, $\theta = 3$, and the hyperedge
appearing at step~$i$ is labeled $e_i$.
\begin{remark}
For $\theta \ge 3$, the hyperrecursive tree is not a tree at all.
We call such a hypergraph by this name only to preserve 
the historic origin
of these structures and frame them as a generalization of genuine recursive trees
($\theta=2$).
\end{remark}
\bigskip
\begin{figure}[thb]
\begin{center}
\begin{tikzpicture}[scale=0.6]

\node[draw=white] at (-10.9, 10) {$e_0$};
\node[draw=white] at (-2.9, 10) {$e_0$};
\node[draw=white] at (5.1, 10) {$e_0$};
\node[draw=white] at (7.25, 6.5) {$e_1$};
\node[draw=white] at (-0.75, 6.5) {$e_1$};
\node[draw=white] at (10.8, 8.1) {$e_2$};

\node[draw=white] at (-8, 10) {$0$};
\node[draw=white] at (-9.5, 10) {$0$};
\node[draw=white] at (-6.5, 10) {$0$};

\draw [ultra thick] (-8,10) ellipse (2.5 and 0.8);
\draw [thick] (-9.5,10) circle [radius=0.4];
\draw [thick] (-8,10) circle [radius=0.4];
\draw [thick] (-6.5,10) circle [radius=0.4];

\node[draw=white] at (0, 10) {$0$};
\node[draw=white] at (-1.5, 10) {$0$};
\node[draw=white] at (1.5, 10) {$0$};
\node[draw=white] at (-0.75, 8) {$1$};
\draw [ultra thick] (0,10) ellipse (2.5 and 0.8);
\draw [ultra thick] (-0.75,9.5) ellipse (1.3 and 2.5);
\draw [thick] (-1.5,10) circle [radius=0.4];
\draw [thick] (0,10) circle [radius=0.4];
\draw [thick] (1.5,10) circle [radius=0.4];
\draw [thick] (-0.75,8) circle [radius=0.4];

\node[draw=white] at (8, 10) {$0$};
\node[draw=white] at (6.5, 10) {$0$};
\node[draw=white] at (9.5, 10) {$0$};
\node[draw=white] at (7.25, 8) {$1$};
\node[draw=white] at (9.25,7.5) {$2$};
\draw [ultra thick] (8,10) ellipse (2.5 and 0.8);
\draw [ultra thick] (7.25,9.5) ellipse (1.3 and 2.5);
\draw [thick] (6.5,10) circle [radius=0.4];
\draw [thick] (8,10) circle [radius=0.4];
\draw [thick] (9.5,10) circle [radius=0.4];
\draw [thick] (7.25,8) circle [radius=0.4];
\draw [thick] (9.25,7.5) circle [radius=0.4];
\draw[ultra thick] (8.7, 8.7) .. controls (9.3, 12.2)
and (10.5, 12.2) .. (10.297, 6.18);

\draw[ultra thick] (8.72, 8.736) .. controls (5.3, 8.5)
and (5.3, 8) .. (10.338, 6.2);

\end{tikzpicture}
\end{center}
  \caption{A hyperrecursive tree grown in two steps ($\theta =3$).}
  \label{Fig:hyper}
\end{figure}
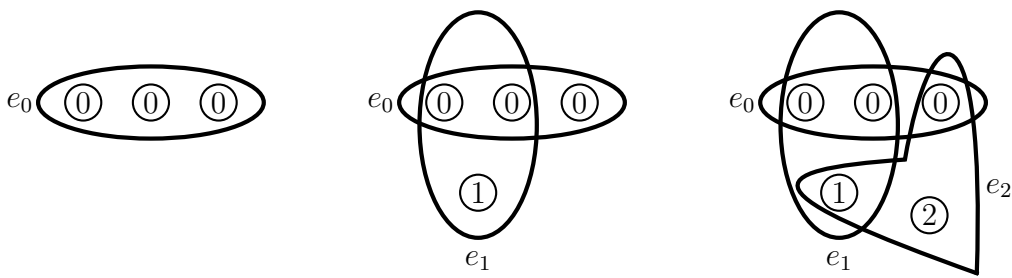
\section{Scope}
Our interest is in a local profile of the {\em level of containment}
for a vertex, which is the number of hyperedges containing it.
Let $C_{n,k}^{(\theta)}$ be the number of hyperedges containing vertex $k\ge 1$ at time $n \ge 0$.
For instance, in the hyperrecursive tree in Figure~\ref{Fig:hyper}, at $n=2$,
the vertex labeled 1 
is at containment level 2, while  the vertex labeled 2 
is at containment level 1, and so  
$C_{2,1}^{(3)} =2$ and $C_{2,2}^{(3)} = 1$.
Each of the originators (all labeled with 0) may evolve differently and have
different levels of containment at time $n$. However, the levels of these originators
have the same distribution at time~$n$, so we can choose a representative
among them to include the case $k=0$. For example,    
the representative of the originators in Figure~\ref{Fig:hyper} may be taken to be the rightmost vertex in the sole hyperedge at time 0. In the series
of hypergraphs shown, the representative experiences the evolution
$C_{0,0}^{(3)} =1$,
$C_{1,0}^{(3)} =1$, and
$C_{2,0}^{(3)} =2$.

We later see that a vertex labeled $k = k(n)$ has 
an asymptotic distribution depending on the relation between
$k$ and $n$. Asymptotically,
earlier vertices have Gaussian distributions and later vertices
have shifted Poisson distributions. The notions of ``early'' and ``late'' are 
made precise in the sequel, with nuances obtained from refinements into very early, early, intermediate,
late, and very late vertices. We also discuss how the phases merge at the seam
lines.  

Moreover, we investigate a global profile of containment. Let 
$X_{n,i}^{(\theta)}$ be the number of vertices at containment level $i$, that is, the number of vertices contained in $i$ edges
after $n$ evolutionary
steps. 
For instance, in the example shown in Figure~\ref{Fig:hyper},
we have $X_{2,1}^{(3)} = 1$, 
$X_{2,2}^{(3)} = 4$, and $X_{2,i} ^{(3)} = 0$, for $i\ge 3$.
To get a glimpse into the interaction between different levels
of containment, we compute the variance-covariance matrix 
of $(X_{n,1}^{(\theta)}, X_{n,2}^{(\theta)})$. Furthermore,  
we develop 
a Gaussian distribution for
$X_{n,1}^{(\theta)}$.
\section{Notation}
\label{Sec:notation}
Let $\tau_n^{(\theta)}$ be the cardinality
 of the vertex set ({\em size}) of the hyperrecursive tree at time~$n$
(right after the insertion of vertex $n$).
We sometimes refer to $n$ as the {\em age} of the hyperrecursive tree. 
Note that
\begin{equation}
\tau_n^{(\theta)} =  n + \theta.
\label{Eq:total}
\end{equation}
The exact results are represented in terms of Pochhammer's symbol for the rising factorial. The $m$-times rising factorial of a real number $x$ is
$$\langle x \rangle_m = x(x+1) \ldots (x+m-1),$$
with the interpretation that $\langle x \rangle_0=1$. We make use of such an expression in its form as a generating function of
the signless Stirling numbers of the first kind, which is namely
\begin{equation}
\langle x \rangle_m = \sum_{i=1}^m \Big[ {m \atop i}\Big] x^i,
\label{Eq:Stirling} 
\end{equation}
where $\big[ {m \atop i}\big]$ is the  $i^{\textnormal{th}}$ (signless) Stirling number of order $m$ of the first kind, a count of the number of permuations of
$\{1, 2, \ldots, m\}$ that have $i$ cycles.   
For properties of Stirling numbers, we refer the reader   
to~\cite{Barton,Graham}.

Average values and variances contain generalized harmonic numbers. These are
$H_n^{(s)} (x)$ $= \sum_{k=1}^n 1/(k+x)^s$, for integer $n\ge 0$ and real $s,x\ge 0$. 
The superscript is often dropped when it is 1;\footnote{The number $H_n^{(1)}(0)$ is 
often written as $H_n$.} we follow this convention.
It is well known that, for any fixed $x$, as  $n \to\infty$, we have
\begin{align}
H_n(x) &= H_n^{(1)}(x) = \ln (n) - \psi(x) - \frac 1 x+O\Big(\frac 1 n\Big);
           \label{Eq:harmonic1}\\ 
H_n^{(s)}(x) &= O(1) \,\, \textnormal{for} \,\, s>1,
\label{Eq:harmonic2}
\end{align}
where $\psi(.)$ is the digamma function. Note that $- \psi(x) - \frac 1 x$ converges to Euler's constant $\gamma$, as $x\to 0$.

In the asymptotic analysis, we utilize the
Stirling approximation of the ratio of growing Gamma functions, as detailed in \cite{Tricomi}. Namely, for fixed~$a$ and $b$ in $\mathbb R$,
we have
\begin{equation}
\frac{\Gamma(x+a)} {\Gamma(x+b)} \sim x^{a-b} \Big(1 +\frac{(a-b)(a+b-1)}{2x}+ O\Big(\frac 1 {x^2} \Big)\Big), \qquad \mbox {as \ } x \to\infty.
\label{Eq:Stirlingapprox}
\end{equation}
This approximation is applicable, even if $a = a(x)$ and $b =b(x)$ 
grow slowly with $x$. 

The sample taken at time $n$ is drawn without replacement, and so the
number of vertices in it at containment levels $1, \ldots, k$  have a (conditional) hypergeometric distribution. 

We use the notation Hypergeo$(\tau, n_1, 
n_2, \ldots, n_r; s)$ for
the multivariate hypergeometric random vector, in which the $i^{\textnormal{th}}$ component
is the number of balls of color $i$ that appear in a sample of size $s$ drawn from
an urn containing~$\tau$ balls, of which $n_i$ balls are of color $i$,
for $i=1,\ldots, r$.  This multivariate  hypergeometric distribution 
is standard
and can be found in classic books on distribution theory, such 
as~\cite{Kendall}. In particular, we need the mean, variances and covariance
for a bivariate marginal distribution. Suppose $Y_i$ is the number
of balls of color $i$, for $i=1, \ldots, r$, that appear in the sample. Then, $(Y_i, Y_j, \tau-Y_i-Y_j)$ have a  trivariate hypergeometric  distribution like Hypergeo$(\tau, n_i, n_j, \tau - n_i - n_j; s)$, with  $Y_i$ distributed like
Hypergeo$(\tau, n_i,  \tau - n_i; s)$.
Later, we utilize the formulas
\begin{align*}
\E[Y_i] &= \frac {n_i} \tau \, s,\\
\V[Y_i] &=   \frac {n_i(\tau-n_i)(\tau-s)} {\tau^2(\tau-1)} \, s,\\
\Cov[Y_i, Y_j] &= - \frac {n_i n_j(\tau-s)} {\tau^2(\tau-1)} \, s;
\end{align*}
see~\cite{Kendall}.

To develop martingale differences, we use the backward difference 
operator~$\nabla$. Acting on a function $h_n$, this operator stands for $\nabla h_n = h_n- h_{n-1}$.  
\section{Local containment profile}
Let $\indicator_{n,k}^{(\theta)}$ be an indicator of the event that vertex~$n$
chooses vertex $k$ in the hyperedge appearing at time $n$. 
The indicator $\indicator_{n,k}^{(\theta)}$ 
is a Bernoulli random variable that assumes 
the value~1 with probability ${\tau_{n-1}^{(\theta)}-1 \choose {\theta -2}} /{\tau_{n-1}^{(\theta)} \choose {\theta -1}} = (\theta-1)/\tau_{n-1}^{(\theta)}$, otherwise it assumes
the value 0 with the complement probability.
This indicator  has the moment generating function
\begin{equation}\psi_{n,k}^{(\theta)} (t) := \E\big[e^{\indicator_{n,k}^{(\theta)}t}\big] = 1- \frac {\theta-1} {\tau_{n-1}^{(\theta)}} 
          + \frac {(\theta-1) e^t} {\tau_{n-1}^{(\theta)}}.
\label{Eq:psi}
\end{equation}

We have a stochastic 
recurrence relation for $C_{n,k}^{(\theta)}$. 
At time $n$, the vertex labeled $k$ either retains its level of containment at time $n-1$ (if it is not chosen for the
$n^{\textnormal{th}}$ hyperedge), or  its level of containment increases by 1 (if chosen for the
$n^{\textnormal{th}}$ hyperedge). We thus have
\begin{equation}
C_{n,k}^{(\theta)} = C_{n-1,k}^{(\theta)} + \indicator_{n,k}^{(\theta)}.
\label{Eq:stoch}
\end{equation}
The earliest time at which vertex $k$ is in the hyperrecursive tree is $k$, at which point 
it is contained in exactly one hyperedge. Therefore, the boundary
condition is $C_{k,k}^{(\theta)} = 1$. Note that $C_{n-1,k}^{(\theta)}$ and
$\indicator_{n,k}^{(\theta)}$ are independent.

Unwinding the recurrence~(\ref{Eq:stoch}) back to the boundary condition, 
we get a representation 
\begin{equation}
C_{n,k}^{(\theta)} = 1 + \indicator_{k+1, k}^{(\theta)}  + \indicator_{k+2,k} ^{(\theta)}+ \cdots+ \indicator_{n,k}^{(\theta)}
\label{Eq:sumindicators}
\end{equation}
into independent (but not identically distributed) indicator random variables. 
\begin{proposition}
\label{Prop:meanvar}
Let $C_{n,k}^{(\theta)}$ be the containment level of the vertex labeled $k\ge 0$ in a 
hyperrecursive tree  of edge size $\theta$ at age $n$.\footnote{Recall that when $k=0$,
we are tracking a chosen representative among the originators (all 
labeled with 0).} 
We have
\begin{align*}
\E\big[C_{n,k}^{(\theta)}\big] &= 1 +  (\theta-1) \big(H_n (\theta-1) - H_k (\theta-1)\big)  \\
       &= \begin{cases}(\theta-1)\big(\ln  (n) -\psi(\theta-1) 
                  - \frac 1 {\theta -1}\big)  \\
         \qquad {}+1 - (\theta-1) H_k (\theta-1)  + O\big(\frac 1 n\big), 
               &\mbox{if \ } k \ge 0 \ \mbox{is fixed};\\
                 (\theta-1) \ln \big(\frac n k \big)    + 1 + O\big(\frac 1 k\big), 
               &\mbox{if \ }n \ge k\to \infty.
\end{cases}
\end{align*}
and
\begin{align*}
\V\big[C_{n,k}^{(\theta)}\big] &= (\theta-1) \big(H_n(\theta-1) - H_k (\theta-1) \big) \\ 
     & \qquad\qquad {} - (\theta-1)^2 \big(H_n^{(2)}
              (\theta-1) - H_k^{(2)}  (\theta-1) \big) \\ 
    &= \begin{cases}(\theta-1)  \ln (n) + O(1), 
               &\mbox{if \ } k \ge 0 \ \mbox{is fixed};\\
               (\theta-1)  \ln \big(\frac n k \big)  + O\big(\frac 1 k\big), 
               &\mbox{if \ } n\ge k\to \infty.
\end{cases}
\end{align*}           
\end{proposition}
\begin{proof}
Taking expectations of~(\ref{Eq:sumindicators}), we find
\begin{align*}
\E\big [C_{n,k}^{(\theta)}\big] 
   &= 1 + \E\big[\indicator_{k+1, k}^{(\theta)}\big]  
       + \E\big[\indicator_{k+2,k}^{(\theta)}\big]
           + \cdots+ \E\big[\indicator_{n,k}^{(\theta)}\big]\\
      &=  1 +  \sum_{i=k+1}^n \frac {\theta-1} {\tau_{i-1}^{(\theta)}}\\
      &=  1+  (\theta-1)\sum_{i=k+1}^n \frac 1 {i + \theta-1}\\
      &=  1 +  (\theta-1) \big(H_n (\theta-1) - H_k  (\theta-1)\big).
\end{align*}
The asymptotic average follows from the approximation in~(\ref{Eq:harmonic1}).

By the independence of the indicators in~(\ref{Eq:sumindicators}), we similarly have
\begin{align*}
\V\big[C_{n,k}^{(\theta)}\big] 
      &= \sum_{i=k+1}^n \V\big[\indicator_{n,k}^{(\theta)}\big] \\
      &= \sum_{i=k+1}^n \frac {\theta-1} {i + \theta-1} 
              - \sum_{i=k+1}^n \frac {(\theta-1)^2} {(i + \theta-1)^2}\\
      &= (\theta-1)\big(H_n (\theta-1) - H_k (\theta-1)\big)  \\
                  &\qquad \qquad {}- (\theta-1)^2 \big(H_n^{(2)}(\theta-1) - H_k^{(2)}(\theta-1)\big).    
\end{align*}
The asymptotic variance follows from the 
approximations in~(\ref{Eq:harmonic1})--(\ref{Eq:harmonic2}).
\end{proof}          
\begin{lemma}
\label{Lem:mgf}
The moment generating function $\phi_{n,k}^{(\theta)}(t) = \E[e^{C_{n,k}^{(\theta)}t}]$ of $C_{n,k}^{(\theta)}$ is given by
$$\phi_{n,k}^{(\theta)} (t)= e^t \prod_{i=k}^{n-1} \frac {i + 1 + (\theta-1)e^t}{i+\theta}. $$
\end{lemma}
\begin{proof}
The representation~(\ref{Eq:sumindicators}) as a sum of independent random variables gives rise to 
\begin{align*}
\phi_{n,k}^{(\theta)}(t)  &= \E\big[e^{C_{n,k}^{(\theta)}t}\big]  \\
     &= \E\big[e^{\big( 1 + \indicator_{k+1, k}^{(\theta)}  + \indicator_{k+2,k} ^{(\theta)}
                + \cdots+ \indicator_{n,k}^{(\theta)}\big)t} \big]\\
     &= e^t\, \E\big[e^{\indicator_{k+1, k}t} \big] 
                            \, \E\big[e^{\indicator_{k+2, k}^{(\theta)}t} \big]
               \ldots \E\big[e^{\indicator_{n, k}^{(\theta)}t}\big ] \qquad 
                                         {\rm (by\ independence)}\\
     &=  e^t \, \psi_{k+1, k}^{(\theta)}(t)\, \psi_{k+2, k}^{(\theta)}(t) 
     \cdots \, \psi_{n, k}^{(\theta)}(t).
\end{align*}    
Plug in~(\ref{Eq:psi}) and the sizes of the hyperrecursive trees in~(\ref{Eq:total}), 
and the statement follows after simplification.
\end{proof}

From Lemma~\ref{Lem:mgf},  we develop an exact distribution.
\begin{theorem}
\label{Theo:exact}
For $n\ge 1$ and $0\le k \le n$, let $C_{n,k}^{(\theta)}$ 
be the level of containment of the vertex $k$ in a hyperrecursive tree 
with edge size $\theta$ at age $n$. 
For $r\ge 1$, we have
$$\Prob\big(C_{n,k}^{(\theta)} = r\big) =
   \frac  {(\theta-1)^{(r-1)}} {\langle k +\theta
                          \rangle_{n-k}} \sum_{i=r-1}^{n-k}  
                            \Big[ {n-k\atop i}\Big]  {i \choose r-1}(k+1)^{i-r+1}.$$                          
\end{theorem} 
\begin{proof}
To obtain a probability generating function $\zeta_{n,k}^{(\theta)}(u) = \sum_{i=0}^\infty \Prob(C_{n,k}^{(\theta)} = i) u^i $, we
substitute $\ln(u)$ for $t$ in the generating function of 
Lemma~(\ref{Lem:mgf}). We obtain
\begin{align*}
\zeta_{n,k}^{(\theta)} (u) &= u\, \frac { (k+1 + (\theta-1) u)\ldots( n  + (\theta-1)  u)} 
    {(k+\theta) \cdots (n+\theta-1)} \\
              &= u\, \frac {\langle k+1 
                  +(\theta-1)u \rangle_{n-k}} {\langle k +\theta
                          \rangle_{n-k}} . 
\end{align*}
Using the generating function~(\ref{Eq:Stirling}), we write 
\begin{align*}
 \zeta_{n,k}^{(\theta)} (u) &= \frac u {\langle k +\theta
                          \rangle_{n-k}} \sum_{i=0}^{n-k}
                              \Big[ {n-k\atop i}\Big]
                           \big(k+ 1 + (\theta-1)u\big)^i \\
         &=  \frac u {\langle k +\theta
                          \rangle_{n-k}} 
                           \sum_{i=0}^{n-k}
                              \Big[ {n-k\atop i}\Big]
                           \sum_{m=0}^i (k+ 1)^{i-m} 
                            (\theta-1)^m u^m {i \choose m}\\
         &= \frac u {\langle k +\theta
                          \rangle_{n-k}}
                           \sum_{m=0}^{n-k} (\theta-1)^m u ^m
                           \sum_{i=m}^{n-k}  
                            \Big[ {n-k\atop i}\Big]  {i \choose m}(k+1)^{i-m} .
                           \end{align*}
The exact distribution in the statement of the theorem
follows upon extracting coefficients.
\end{proof}
\subsection{Phases in the local containment profile of a vertex}
\label{Subsec:phases}
Proposition~\ref{Prop:meanvar} shows that
the mean and variance of the hyperrecursive trees experience a phase change, 
as $k$ increases relative to $n$. For instance, for fixed~$k$,
the mean is asymptotic to $(\theta-1) \ln (n)$, as $n\to \infty$, and $k$ can 
only alter lower-order asymptotics. 
Such is the case for all fixed $k$, as $n\to\infty$. However,
a phase transition occurs when $k$ grows to infinity with~$n$, but remains $o(n)$, such as the case
$k = k(n) = \lceil  3 n^{\frac 3 4} + 2  \sqrt n - \pi \rceil$. In this range, $k$ provides essential leading-term asymptotics. The vertices that appear in the entire range in which $k = o(n)$ are 
``early.'' 

For the linear ``intermediate'' range, $k \sim \alpha n$, for $0 < \alpha < 1$, 
such as the case $\lceil \frac 5 7 n + 3\sqrt n + 6e\rceil$, 
the asymptotic mean is $1 + (\theta-1) \ln \big(\frac 1 \alpha \big)$. 
Vertices in the range $j \sim n$ are considered ``late''. 
The asymptotic mean of late vertices is just $1$, showing that the 
late arrivals, such as the case $k = \lfloor n - 5\ln (n+1) + 14\rfloor$, have negligible probability of participating in the recruiting events.

It is possible to conceive of a bizarre relation between $k$ and $n$, 
such as, for example, $k= \lfloor (\frac 1 2+ \frac {(-1)^n} 3) n \rfloor$, for which the 
mean containment level oscillates, without settling on any asymptotic average.
In these cases, we have no convergence in the mean, variance or distribution.
Such an oscillating case is not likely to appear in practice. 
 
Naturally, these phases in the mean are reflected in the asymptotic distributions. 
From Lemma~\ref{Lem:mgf}, we can get asymptotic distributions. 
It is beneficial for the asymptotic analysis 
to represent the product in 
Lemma~\ref{Lem:mgf} in terms of Gamma functions:
\begin{equation}
\phi_{n,k}^{(\theta)} (t) = e^{t} \, \frac { \Gamma( n  + 1 + (\theta-1)e^t)\, \Gamma (k+\theta)} 
    {\Gamma(n+ \theta) \,\Gamma(k+1+(\theta-1)  {e^t})}.
    \label{Eq:usher} 
\end{equation} 
\begin{theorem}
\label{Theo:phases}
Let $0\le k \le n$ and $C_{n,k}^{(\theta)}$ be the containment 
level of the vertex~$k$ in a hyperrecursive tree with edge size $\theta$ at age $n$. We have
\begin{itemize}
\item [(i)] For $k$  fixed: 
$$\frac{C_{n,k}^{(\theta)} -  (\theta-1) \ln (n)} {\sqrt {\ln (n)}}
    \ \convD\  \normal \big(0,\theta-1\big), $$
\item [(ii)]   For $k\to \infty$ and $k = o(n)$:
 $$\frac{C_{n,k}^{(\theta)} -  (\theta-1) \ln \big(\frac n k \big)} {\sqrt {\ln \big(\frac n k \big)}}
    \ \convD \ \normal \big(0,\theta-1\big), $$
\item [(iii)]   For $k\sim \alpha n$, and $0 < \alpha < 1$:
 $$C_{n,k}^{(\theta)} \
    \ \convD\  \ 1 +  \Poi \Big((\theta-1) \ln \Big(\frac 1 \alpha\Big)\Big),$$
 \item [(iv)]  For $k = n- o(n)$:  
  $$C_{n,k}^{(\theta)} \ \inprob \  1.$$
 \end{itemize}
\end{theorem}
\begin{proof}
We start with the phase in which $k$ is fixed.
For this case, by the Stirling approximation in~(\ref{Eq:Stirlingapprox}), 
we write~(\ref{Eq:usher}) with $t$ scaled:
\begin{align*} \phi_{n,k}^{(\theta)} \Big( \frac t {\sqrt {\ln (n)}}\Big) 
    &=e^{t/\sqrt {\ln( n)}} \, \frac { \Gamma\big( n  + 1 + (\theta-1)e^{t/\sqrt {\ln (n)}}\big) \,\Gamma (k + \theta)} 
    {\Gamma (n+\theta)\, 
            \Gamma\big(k +1+ (\theta-1)e^{t/\sqrt {\ln (n)}} \big)}\\
    &\sim  {n^{(\theta-1)(e^{t/\sqrt {\ln (n)}} - 1)}}.
\end{align*}
Going through a local expansion of the exponential, we write
\begin{align*} \phi_{n,k}^{(\theta)} \Big( \frac t {\sqrt {\ln (n)}}\Big) 
    &\sim  e^{\big((\theta-1) \big( 1 + \frac 1 {\sqrt{\ln (n)}} 
       +  \frac {t^2} {2 \ln (n)} +  O\big(\frac {t^3} {\ln (n)}\big) \big)
               \ln (n)- (\theta-1)\ln (n)\big)}.
\end{align*}
We can reorganize this relation as
\begin{align*} \phi_{n,k}^{(\theta)} \Big( \frac t {\sqrt {\ln (n)}}\Big)
   e^{-  (\theta-1) \, t \sqrt{\ln (n)}}
    &\sim  \exp\Big(\frac {(\theta-1)t^2} 2 +  O\Big(\frac {t^3} {\sqrt{\ln (n)}}\Big) \Big).
\end{align*}
At any fixed $t\in \mathbb R$,
we have convergence
$$\E\Big[e^{\frac{C_{n,k}^{(\theta)} -  (\theta-1) \ln (n)} {\sqrt {\ln (n)}}\, t} \Big]
    \to  e^{\frac {(\theta-1) t^2} 2}.$$
The right-hand side is the moment generating function of a centered normal distribution with variance $\theta-1$.
By L\'evy's continuity theorem~\cite{Williams}, we establish convergence in distribution as stated in Part (i).
    
The analysis of the rest of the phase of early $k$, a phase in which $k\to\infty$, but $k=o(n)$, is not much different from
the fixed $k$ phase. It only requires
some minor tweaks to bring in the role of $k$, which is now pronounced.
In this phase, we apply the Stirling approximation in~(\ref{Eq:Stirlingapprox}) to all four gamma functions in~(\ref{Eq:usher}). Consequently, we have
\begin{align*} 
\phi_{n,k}^{(\theta)} \Big( \frac t {\sqrt {\ln (n/k)}}\Big) 
    &= e^{t/\sqrt {\ln (n/k)}} \, \frac {n^{ 
         (\theta-1)t/\sqrt{\ln (n/k)}
        - (\theta-1)}  \big(1+O\big(\frac 1 n\big)\big)} 
                {k^{ (\theta-1)e^{t/\sqrt{\ln (n{\bf red} /k )}} - (\theta-1)} 
         \big(1+O\big(\frac 1 k\big)\big)}\\
    &\sim \Big(\frac n k \Big)^{(\theta-1)e^{t/ \sqrt{\ln (n/k)}}  
          - (\theta-1)}.
\end{align*}
From here, steps follow as in the case of fixed $k$. We get convergence in distribution as stated in Part (ii) of the theorem.

In the intermediate and late phases $j\sim \alpha n$, for $\alpha \in (0, 1]$, no scaling is required to get convergence in distribution. Instead, we have 
\begin{align*}
\phi_{n,k}^{(\theta)}(t)
    &= e^t \, \frac {n^{ (\theta-1)e^t - (\theta-1)}  \big(1+O\big(\frac 1 n\big)\big)} 
               {k^{ (\theta-1)e^t - (\theta-1)} 
         \big(1+O\big(\frac 1 k\big)\big)}\\
    &\to  e^t \Big(\frac 1 \alpha\Big)^{(\theta-1)(e^t - 1)}\\
    &=  e^{t +  (\theta-1)\ln (\frac 1 \alpha) \, ( e^t- 1)}.\\
\end{align*}  
The moment generating function on the right-hand side is that of 1 added to a Poisson random variable 
with mean $(\theta-1) \ln(1/\alpha)$.  
By L\'evy's continuity theorem (Theorem 18.1 in~\cite{Williams}), we establish convergence in distribution as stated in Part (iii).
Then, $C_{n,k}^{(\theta)}$ degenerates to a constant in the case $\alpha =1$, where we get
$$C_{n,k}^{(\theta)}
     \ \convD\  1 .$$ 
Convergence in distribution to a constant implies convergence in probability, 
as stated in Part (iv) of the theorem. 
\end{proof}
\begin{remark}
Phase changes of the type in Theorem~\ref{Theo:phases} have been observed
in~\cite{QQ} and~\cite{Mah2019}.
\end{remark}
\begin{remark}
At the seam lines between the phases, the change is not abrupt. In fact,
the phases ``flow into each other'' in a natural way. For instance, in the case
of Part (ii), we can write $\ln(n/k)$ as $\ln (n) - \ln (k)$. Then, we see that
for $k$ fixed, $\ln (k) / \sqrt {\ln (n)} \to 0$, and by Slutsky's 
theorem~\cite{Slutsky}, we get
the statement adjusted as in Part (i). 
The role of $k$ is negligible, so long as $\ln (k) = o (\sqrt{ \ln (n)}\,)$.
Past this threshold,  $\ln (k)$ cannot be neglected relative to $\sqrt{ \ln (n)}$, 
and $k$ must be included in the convergence.
Again, the limit in
Part (iii) converges to that in Part (iv), as $\alpha \to 1$.      
\end{remark}   
\section{Global containment profile}
In this section, we look at a profile of the hyperrecursive tree determined by
a raw count of the number of vertices at a particular containment level.
Such a profile is global, as it cannot be determined without looking at
all the vertices in the entire hyperrecursive tree.

We defined  $X_{n,i}^{(\theta)}$ to be the number of vertices contained in exactly $i$ hyperedges.
To discern how the different containment  levels interact,
we investigate the mean of the row vector $(X_{n,1}^{(\theta)}, X_{n,2}^{(\theta)})$ and
its covariance matrix.  

We start with stochastic recurrences from which we proceed to a calculation 
of the exact mean and variances, which lead us to concentration laws. 
Eventually, we establish a central limit theorem for the vertices at the smallest 
level of containment, i.e., the vertices contained in one hyperedge. Note 
that when
$\theta=2$, in which case the hypergraph becomes the uniform
recursive tree, vertices at containment level 1 are simply
the leaves.
\subsection{Stochastic recurrences}
We discuss here recurrence equations that hold on the stochastic
path. Let~$Q_{n,i}^{(\theta)}$ be the number of vertices at containment level $i=1,2$
that appear in the sample chosen to construct the $n^{\textnormal{th}}$ hyperedge.
The row vector 
$$\big(Q_{n,1}^{(\theta)}, Q_{n,2}^{(\theta)}, \tau_{n-1}^{(\theta)} - Q_{n,1}^{(\theta)}- Q_{n,2}^{(\theta)}\big)$$ 
has a (conditional) trivariate
hypergeometric distribution, that selects
a sample of size $\theta-1$ from among $\tau_{n-1}^{(\theta)}$ vertices, of which
$X_{n,i}^{(\theta)}$ are at containment level $i$, for $i=1, 2$.
That is, the components of this row vector have the conditional joint
distribution
$$\Prob\big(Q_{n,1}^{(\theta)} =q_1, Q_{n,2}^{(\theta)}=q_2
   \, \big | \, X_{n-1,1}^{(\theta)},  X_{n-1,2}^{(\theta)}\big) 
   = \frac {{X_{n-1,1}^{(\theta)} \choose q_1} {X_{n-1,2}^{(\theta)}
 \choose q_2} {\tau_{n-1}^{(\theta)} - X_{n-1,1}^{(\theta)} -
    X_{n-1,2}^{(\theta)} \choose \theta -1 - q_1 - q_2}}{{\tau_{n-1}^{(\theta)}\choose \theta-1}}.$$  
The conditional means, variances, and the covariance are specified
in Section~\ref{Sec:notation}.

In the construction of the $n^{\textnormal{th}}$ hyperedge, each vertex at containment level~1 in the sample becomes upgraded
to containment level 2, and the newly added vertex at step $n$ is
at containment level 1. Whence, we have the stochastic recurrence
\begin{equation}
X_{n,1}^{(\theta)} = X_{n-1,1}^{(\theta)} - Q_{n,1}^{(\theta)} + 1.
\label{Eq:stoch1}
\end{equation}
Each vertex at containment level 2 in the sample becomes upgraded
to containment level 3. However, the $Q_{n,1}^{(\theta)}$ vertices 
at containment level 1 in the sample all become at containment level 2, giving rise to
the stochastic recurrence
\begin{equation}
X_{n,2}^{(\theta)} = X_{n-1,2}^{(\theta)} - Q_{n,2} ^{(\theta)}+  Q_{n,1}^{(\theta)}.
\label{Eq:stoch2}
\end{equation}
\subsection{The mean and covariance matrix}
The pair of stochastic equations~(\ref{Eq:stoch1})--(\ref{Eq:stoch2})
is sufficient to determine the means exactly and the quadratic order moments
asymptotically.
\begin{proposition}
\label{Prop:mean}
Let $X_{n,i}^{(\theta)}$ be the number of vertices  contained in exactly $i$ hyperedges, for $i=1,2$, of a recursive hyperrecursive tree of edge size $\theta$. We have the mean vector
\begin{align*}
\begin{pmatrix}
     \E\big[X_{n,1}^{(\theta)}\big]\\
     \E\big[X_{n,2}^{(\theta)}\big]
   \end{pmatrix}  &= \begin{pmatrix} \displaystyle
     \frac 1 \theta\,  n + 1 + (\theta-1) \, \Gamma (\theta) \frac {\Gamma(n+1)}{
                \Gamma(n+\theta)} \\ \\
     \displaystyle  \frac {(\theta-1)} {\theta^2} ( n + \theta) 
              +   \Gamma (\theta) \frac {\Gamma(n+1)}{
                \Gamma(n+\theta)} \Big((\theta-1)^2 H_n - \frac {\theta -1}
                       \theta \Big) 
   \end{pmatrix} \\
   &= \begin{pmatrix}  \displaystyle \frac n \theta +1+ O  \Big(\frac 1 {n^{\theta-1}} \Big) \\ \\
   \displaystyle  \frac {(\theta-1)} {\theta^2} (  n + \theta)  
        + O\Big(\frac {\ln (n)} {n^{\theta-1}}\Big)
   \end{pmatrix}.
\end{align*}
\end{proposition} 
\begin{proof}
Let $\field_n$ be the sigma field generated by the first $n$ hyperedge additions. Conditioning the stochastic relations~(\ref{Eq:stoch1}) 
and~(\ref{Eq:stoch2}) on $\field_{n-1}$, we obtain
 \begin{align*}
\E\big[X_{n,1}^{(\theta)} \, |\, \field_{n-1}\big] =& X_{n-1,1} - \E\big[Q_{n,1}^{(\theta)} \, |\, \field_{n-1}\big] + 1,\\
\E\big[X_{n,2}^{(\theta)}  \, |\,\field_{n-1}\big] =& X_{n-1,2}^{(\theta)} - \E\big[Q_{n,2}^{(\theta)} \, |\, \field_{n-1}\big] +  \E\big[Q_{n,1}^{(\theta)} \, |\, \field_{n-1}\big].
\end{align*} 
As discussed, the random variables $Q_{n,1}^{(\theta)}$,  $Q_{n,2}^{(\theta)}$ ,  and $\tau_{n-1}^{(\theta)} - Q_{n,1}^{(\theta)} - Q_{n,2}^{(\theta)}$
have a (conditional) trivariate hypergeometric marginal distribution; the conditional 
expectations are
 \begin{align}
\E\big[X_{n,1}^{(\theta)} \, |\, \field_{n-1}\big] &= X_{n-1,1}^{(\theta)} -\frac {X_{n-1,1}^{(\theta)}}{\tau_{n-1}^{(\theta)}} (\theta-1) +1, \label{Eq:X1}\\
\E\big[X_{n,2}^{(\theta)}\, |\, \field_{n-1}\big] &= X_{n-1,2}^{(\theta)}  -\frac {X_{n-1,2}^{(\theta)}}{\tau_{n-1}^{(\theta)}} (\theta-1) +  \frac {X_{n-1,1}^{(\theta)}}{\tau_{n-1}^{(\theta)}} (\theta-1);\nonumber
\end{align}
see the formulas in Section~\ref{Sec:notation}.
Taking an iterated average, 
 simultaneous recurrences can be written:
 \begin{align}
\E\big[X_{n,1}^{(\theta)}\big] &= \Big(1-  \frac {\theta-1}
      {\tau_{n-1}^{(\theta)}} \Big) 
          \E\big[X_{n-1,1}^{(\theta)}\big]  +1 \nonumber\\
          &=  \frac n {n+\theta-1} 
          \E\big[X_{n-1,1}^{(\theta)}\big]  +1,\label{Eq:EX1}\\
\E\big[X_{n,2}^{(\theta)}\big] &=  \Big(1 -\frac {\theta-1)} {\tau_{n-1}^{(\theta)}}\Big) 
               \E[X_{n-1,2}^{(\theta)}] +  \frac {\theta-1}{\tau_{n-1}^{(\theta)}} \E\big[X_{n-1,1}^{(\theta)}\big]\nonumber\\
                      &=    \frac n {n+\theta-1} 
          \E\big[X_{n-1,2}^{(\theta)}\big]  +  \frac {\theta-1}{n + \theta-1} 
          \E\big[X_{n-1,1}^{(\theta)}\big].
          \label{Eq:Xn2}
\end{align}

The first of these two equations is self contained, while the second has to await
for the solution of the first to be bootstrapped into it. The first equation
has the standard form 
\begin{equation}
y_n = g_n y_{n-1} + h_n,
\label{Eq:standard}
\end{equation}
with solution
$$y_n = \sum_{i=1}^n h_i\prod_{j=i+1}^n g_j + y_0 \prod_{j=1}^n g_j. $$

We solve~(\ref{Eq:EX1}), with $g_n =\frac{ n}{(n+\theta-1)}$ and $h_n =1$, and obtain
\begin{align*}
\E\big[X_{n,1}^{(\theta)}\big] 
    &= \sum_{i=1}^n \prod_{j=i+1}^n \frac j {j+\theta-1} 
         + \theta \prod_{j=1}^n \frac j {j+\theta-1} \\
    &=  \frac {\Gamma(n+1)} {\Gamma(n+\theta)}
   \sum_{i=1}^n  \frac {\Gamma(i+\theta)} {\Gamma(i+1)} + \frac {\theta\, \Gamma(\theta)\, \Gamma(n+1)} 
   {\Gamma(n+\theta)}. 
\end{align*}
To simplify the sum,  we use a known identity, which is namely
\begin{align}
\sum_{i=1}^n\frac{\Gamma(i+\alpha)}{\Gamma(i+\beta)}&= \frac{ \Gamma(n+\alpha+1)}{(\alpha-\beta +1)\Gamma(n+\beta)}  - \frac{ \Gamma(\alpha+1)}{(\alpha-\beta +1)\Gamma(\beta)},
\label{Eq:ratio}
\end{align}
for any given $\alpha, \beta \in \mathbb{R}^+$, such that $\beta \not = \alpha+1$.
Applying this identity with $\alpha=\theta\ge 2$ and $\beta = 1$, we obtain the stated
result after some straightforward simplification.

The asymptotic formula given
is a consequence of the Stirling approximation of the ratio of gamma 
functions  (cf.~(\ref{Eq:Stirlingapprox})). 

With $\E[X_{n,1}^{(\theta)}]$ in hand, we can bootstrap it into the recurrence for 
$\E[X_{n,2}^{(\theta)}]$ to also put that recursion in the form~(\ref{Eq:standard}).
The solution follows similar steps as those used in solving the recurrence
for $\E[X_{n,1}^{(\theta)}]$, and we only highlight the chief steps.
The recurrence~(\ref{Eq:Xn2}) is
\begin{align*}
  \E \big[X_{n,2}^{(\theta)} \big]
&= \frac{n}{n+\theta-1} \,\E \big[X_{n-1,2}^{(\theta)}  \big] \\
&\qquad {}+ \Big(\frac{\theta -1}{n+\theta-1} \Big) \Big(\frac{n-1}{\theta} +1+ \frac{(\theta-1) \, \Gamma(\theta)\, \Gamma(n)}{\Gamma(n+\theta-1)} \Big)   \nonumber \\
&=\frac{n}{n+\theta-1}  \, \E \big[X_{n-1,2}^{(\theta)}  \big] + \Big(\frac{\theta -1}{\theta}+  \frac{(\theta-1)^2\, \Gamma(\theta)\, \Gamma(n)}{\Gamma(n+\theta)} \Big) .
\end{align*}
Here, we have for~(\ref{Eq:standard})
$$y_0=\E\big[X_{0,2}^{(\theta)}\big]=0, \quad g_n= 
   \frac n {n+\theta-1}, \quad  h_n=\Big(\frac{\theta -1}{\theta}+  \frac{(\theta-1)^2\, \Gamma(\theta)\, \Gamma(n)}{\Gamma(n+\theta)} \Big) .$$
An exact solution follows after simplification.
The asymptotic formula given
is a consequence of the Stirling approximation of the ratio of gamma functions (cf.~(\ref{Eq:Stirlingapprox})). 
\end{proof}
\begin{theorem}
\label{Theo:variance}
Let $X_{n,i}^{(\theta)}$ be the number of vertices contained in exactly $i$ hyperedges in a hyperrecursive tree with parameter $\theta$ at age $n$, for $i=1,2$. 
Let ${\bf \Sigma}_n $ be the corresponding covariance matrix.
Upon scaling by $n$, the covariance matrix converges (as $n\to\infty$) as given below:
$$\frac 1 n   {\bf \Sigma_n}  \to 
  \begin{pmatrix}\displaystyle 
   \frac{(\theta-1)^2}{\theta^2(2\theta-1)} 
              &\displaystyle  - \frac{(\theta-1)^2(\theta^2+2\theta-1)}{\theta^3(2\theta-1)^2}\\\\
    \displaystyle - \frac{(\theta-1)^2(\theta^2+2\theta-1)}{\theta^3(2\theta-1)^2} &\displaystyle  \frac{(\theta-1)^2(6\theta^4-6\theta^3+8\theta^2-5\theta+1)}{\theta^4(2\theta-1)^3}
   \end{pmatrix}.$$ 
\end{theorem}
\begin{proof}
It is folklore that variance computation is very lengthy, a phenomenon
called the {\em combinatorial explosion}. We only highlight the salient
points.
  
The starting point for variance-covariance computation is the pair of stochastic
recurrences~(\ref{Eq:stoch1}) and~(\ref{Eq:stoch2}), from which
we can get stochastic recurrence relations for the second-order moments. We take the square of each of these equations, as well as their product: 
\begin{align*}
\big(X_{n,1}^{(\theta)}\big)^2 &= \big( X_{n-1,1}^{(\theta)}\big)^2 + \big(Q_{n,1}^{(\theta)}\big)^2 + 1 - 2 X_{n-1,1}^{(\theta)} Q_{n,1}^{(\theta)} + 2X_{n-1,1}^{(\theta)}  - 2Q_{n,1}^{(\theta)};\\
X_{n,1}^{(\theta)} X_{n,2}^{(\theta)}&= X_{n-1,1}^{(\theta)} X_{n-1,2}^{(\theta)} - X_{n-1,1}^{(\theta)} Q_{n,2}^{(\theta)}
             + X_{n-1,1}^{(\theta)} Q_{n,1}^{(\theta)}\\
&\qquad {} -  X_{n-1,2}^{(\theta)} Q_{n,1}^{(\theta)} + Q_{n-1,1}^{(\theta)} Q_{n,2}^{(\theta)} - \big(Q_{n,1}^{(\theta)}\big)^2 \\
&\qquad {}   +X_{n-1,2}^{(\theta)} -Q_{n,2}^{(\theta)}  + Q_{n,1}^{(\theta)} ;\\
(X_{n,2}^{(\theta)}\big)^2 &= \big(X_{n-1,2}^{(\theta)}\big)^2    + \big(Q_{n,2} ^{(\theta)}\big)^2+  \big(Q_{n,1} ^{(\theta)}\big)^2
    - 2X_{n-1,1}^{(\theta)} Q_{n,2}^{(\theta)}\\
    &\qquad {}  + 2X_{n-1,2}^{(\theta)} Q_{n,1}^{(\theta)} - 2Q_{n,1} ^{(\theta)} Q_{n,2}^{(\theta)}.
\end{align*}

We next take the expectation (conditioned on $\field_{n-1}$)  and use the 
conditional trivariate hypergeometric distribution of $(Q_{n,1}^{(\theta)}, 
Q_{n,2}^{(\theta)},  \tau_{n-1}^{(\theta)}- Q_{n,1}^{(\theta)}-Q_{n,2}^{(\theta)})$, which
comes in terms of $(X_{n-1,1}^{(\theta)}, X_{n-1,2}^{(\theta)})$. 
So, an iterated 
expectation on each conditional recurrence gives us three unconditional
recurrence equations in $\E[(X_{n,1}^{(\theta)})^2]$,
$\E[X_{n,1}^{(\theta)}X_{n,2}^{(\theta)}]$
and $\E[(X_{n,2}^{(\theta)})^2]$.  

 It is evident that we need a bootstrapping technique: the recurrence equation

 for $\E\big[(X_{n,1}^{(\theta)})^2\big]$ is self contained, going back only to $\E\big[(X_{n-1,1}^{(\theta)})^2\big]$. So, we can start with it. It has the form~(\ref{Eq:standard}), with solution
$$\E\big[\big(X_{n,1}^{(\theta)}\big)^2\big] = \frac {n^2}{\theta^2} +\frac {5\theta^2-4\theta+1}{\theta^2(2\theta-1)}\, n  + O(1).$$
 
The recurrence equation for  
$\E\big[X_{n,1} ^{(\theta)} X_{n,2} ^{(\theta)}\big]$ involves $\E\big[(X_{n-1,1}^{(\theta)})^2\big]$, which is now available.
So, after all, the equation is of the form~(\ref{Eq:standard}) 
(with some of the terms in $h_n$  specified only asymptotically). 
We obtain the asymptotic solution
 $$\E\big[X_{n,1} ^{(\theta)} X_{n,2}^{(\theta)}\big] = \frac {\theta -1}{\theta^2}\, n^2  + \frac {7\theta^4-16\theta^3+14 \theta^2 -6\theta +1}{\theta^3(2\theta-1)^2}\, n + O\big(\ln (n)\big).$$

Similarly, the recurrence equation for  
$\E\big[(X_{n,2}^{(\theta)})^2\big]$ involves $\E\big[(X_{n-1,1}^{(\theta)})^2\big]$, as well as $\E\big[X_{n-1,1}^{(\theta)}X_{n-1,2}^{(\theta)}\big]$, which are both available now.
Again, the equation is of the form~(\ref{Eq:standard}) 
(with $h_n$ specified only asymptotically). This gives the solution  
 $$\E\big[(X_{n,2}^{(\theta)})^2\big] = \frac {(\theta -1)^2}{\theta^4}\,n^2  + \frac {(22\theta^4-30\theta^3 + 20 \theta^2 -7\theta +1)(\theta-1)^2}{\theta^4(2\theta-1)^3}\,n +O\big(\ln (n)\big).$$
 
Toward the variance, from the second moments we subtract the expectations of the squares of
the first moments, and toward the covariance, we subtract $\E[X_{n,1}^{(\theta)}]\, \E[X_{n,2}^{(\theta)}]$ from the mixed moment. 
Huge cancellations take place, removing the $n^2$ term from the expression and leaving the covariance matrix convergence as stated.
We relegate all the details to the appendix.
\end{proof} 
\subsection{Concentration laws}
We determine approximations of 
$X_{n,1}^{(\theta)}$ and $X_{n,2}^{(\theta)}$ by the leading asymptotic
equivalents of their means. Errors in the $\OL$ sense\footnote{
		A sequence of random variables $Y_n$
		is $\OL  (g(n))$, when there exist a positive
		constant~$A$ and a positive integer $n_0$, such that $\E[|Y_n|] \le A |g(n)|$, for all $n \ge n_0$.} are of lower order.
\begin{lemma}
	\label{Lm:Lconcentration}
As $n\to\infty$, we have the asymptotic approximation
$$
\begin{pmatrix}
  X_{n,1}^{(\theta)}\\ \\
     X_{n,2}^{(\theta)}
   \end{pmatrix}  =  
   \begin{pmatrix}\displaystyle \frac n \theta  +  \OL \big(\sqrt n\, \big)\\
   \\ \displaystyle
   \frac {(\theta-1) n} {\theta^2}
       +  \OL \big(\sqrt n\, \big)
       \end{pmatrix}.
$$
\end{lemma}
\begin{proof}
	From the asymptotics of the mean and variance, as given in 
	Proposition~\ref{Prop:mean} and Theorem~\ref{Theo:variance},
	we have
	\begin{align*}
	\E\Big[\Big(X_{n,1}^{(\theta)} -  \frac n \theta \Big)^2\Big] 
	&= \E\Big[\Big(\big(X_{n,1}^{(\theta)} 
	    -  \E\big[X_{n,1}^{(\theta)}\big]\big) + \Big(\E\big[X_{n,1}^{(\theta)}\big] - \frac n \theta\Big)\Big) ^2\Big]\\
	&=\V\big[X_{n,1}^{(\theta)}\big] + \Big(\E\big[X_{n,1}^{(\theta)}\big] -\frac n \theta \Big)^2\\
	&= O(n).
	\label{Eq:Lonenorm}
	\end{align*}
	So, by Jensen's inequality
	$$\E\Big[\Big | X_{n,1}^{(\theta)} - \frac n \theta \Big |\Big]
	\le \sqrt {\E\Big[\Big(X_{n,1}^{(\theta)}-  \frac n \theta  \Big)^2\Big] }
	=  O \big(\sqrt n\, \big).$$
	It follows that
	$$X_{n,1}^{(\theta)} =  \frac n \theta  +  \OL \big(\sqrt n\, \big).$$
	The proof of the asymptotic approximation for $X_{n,2}^{(\theta)}$
	is quite similar. 
\end{proof}
\begin{corollary}
As $n\to\infty$, we have
$$\frac 1 n \begin{pmatrix}
   X_{n,1}^{(\theta)}\\ \\
     X_{n,2}^{(\theta)}
   \end{pmatrix} \inprob 
   \begin{pmatrix}\displaystyle
   \frac 1 \theta\\ \\
   \displaystyle
     \frac {\theta-1} {\theta^2}
   \end{pmatrix}.$$
\end{corollary}

\subsection{Martingalization}
We perform a martingale transform on  $X_{n,1}^{(\theta)}$. 
Let $M_n^{(\theta)} = r_n^{(\theta)} X_{n,1}^{(\theta)}+ s_n^{(\theta)}$, for deterministic, but yet-to-be specified, factors $r_n^{(\theta)}$
and $s_n^{(\theta)}$ that render~$M_n^{(\theta)}$ a martingale. Toward
such martingalization, using~(\ref{Eq:X1}), we write
\begin{align*}
\E\big[M_n^{(\theta)}\,\big |\,  \field_{n-1}\big] 
    &= \E\big[r_n^{(\theta)} X_{n,1}^{(\theta)}
    + s_n^{(\theta)}\,\big |\field_{n-1}\big]  \\
    &= \Big(1-\frac{\theta-1}{\tau_{n-1}^{(\theta)}}\Big) r_n ^{(\theta)}
     X_{n-1,1}^{(\theta)}+r_n ^{(\theta)}+ s_n ^{(\theta)}\\
            &= M_{n-1}^{(\theta)}\\
            &= r_{n-1}^{(\theta)} X_{n-1,1}^{(\theta)} + s_{n-1}^{(\theta)}.                 
\end{align*}
This is possible, if
\begin{align*}
 r_{n-1}^{(\theta)} &= \Big(1 - \frac{\theta-1} {\tau_{n-1}^{(\theta)}}\Big) r_n^{(\theta)}; \\
 s_{n-1} ^{(\theta)}&=  r_n^{(\theta)} + s_n^{(\theta)}. 
\end{align*}
The factor $r_n^{(\theta)}$ should satisfy the recurrence                                
$$ r_n^{(\theta)} = \Big(\frac {\tau_{n-1}^{(\theta)}} {\tau_{n-1}^{(\theta)}
 -\theta+1} \Big) r_{n-1} 
   =  \Big(\frac {n+\theta-1} n \Big)r_{n-1}^{(\theta)} , $$
   which unwinds into
 $$ r_n^{(\theta)} = \frac {(n+\theta-1) (n+\theta-2)  \cdots \theta}{n(n-1)\times \cdots
    \times 1}\, r_0^{(\theta)}  = \frac {\Gamma(n+\theta)}
        {\Gamma (\theta)\, \Gamma(n+1)}\, r_0^{(\theta)}, $$
        for any arbitrary $r_0^{(\theta)}\in \mathbb R$;   
        for simplicity, we take $r_0^{(\theta)} =1$.

The factor $s_n^{(\theta)}$ should satisfy the  recurrence 
$$s_n^{(\theta)} =  s_{n-1}^{(\theta)} - r_{n}^{(\theta)}  ,$$
which unwinds into
$$s_n^{(\theta)} = s_0^{(\theta)}- \sum_{i=1}^{n} r_i^{(\theta)},$$
for any arbitrary $s_0^{(\theta)} \in \mathbb R$; 
we take $s_0^{(\theta)} = 0$.
  Using the identity~(\ref{Eq:ratio}) once again, 
we simplify $s_n^{(\theta)}$ to
$$s_n^{(\theta)} = - \frac { \Gamma(n+\theta+1)}{\theta\, \Gamma(\theta)\,\Gamma(n+1) }
    +1.$$

Thus, we have
$$ M_n^{(\theta)} = \frac {\Gamma(n+\theta)}
        {\Gamma (\theta)\, \Gamma(n+1)} \, X_{n,1}^{(\theta)} - \frac {\Gamma(n+\theta+1)}{\theta\,\Gamma(\theta)\,\Gamma(n+1)} + 1$$
        is a martingale.
        
Asymptotics of  $r_n^{(\theta)}$, $s_n^{(\theta)}$ and their backward
differences are useful in the ensuing analysis.
 \begin{lemma}
 \label{Lm:rs}
 As $n\to\infty$, we have the asymptotics:
 $$  r_n^{(\theta)} \sim    \frac {n^{\theta-1}}
        {\Gamma (\theta)},  
        \qquad r_n^{(\theta)}-r_{n-1}^{(\theta)} = O (n^{\theta-2});$$
        $$  s_n^{(\theta)} \sim    - \frac {n^\theta}
        {\theta \, \Gamma(\theta)},  
        \qquad  s_n^{(\theta)}-s_{n-1}^{(\theta)} = O (n^{\theta-1}).$$
 \end{lemma}
 \begin{proof}      
          Examine the forms of 
          $r_n^{(\theta)}$ and $s_n^{(\theta)}$. Their asymptotic equivalents
          follow from the Stirling approximation in~(\ref{Eq:Stirlingapprox}).
          
          Further, we have
\begin{align*}
r_n^{(\theta)} - r_{n-1}^{(\theta)} &=   
        \frac {\Gamma(n+\theta)}
        {\Gamma (\theta)\, \Gamma(n+1)} -  \frac {\Gamma(n-1+\theta)}
        {\Gamma (\theta)\, \Gamma(n)} \\
        &=  \frac {\Gamma(n-1+\theta)}
        {\Gamma (\theta)\, \Gamma(n)} \Big(\frac{n-1+\theta} n -1\Big)\\
        &\sim  \frac{\theta-1} {\Gamma (\theta)} \,n^{\theta-2} , 
        \end{align*}
        and
\begin{align*}
s_n^{(\theta)} - s_{n-1}^{(\theta)} &=   
       - \frac {\Gamma(n+\theta+1)}
       {\theta\,\Gamma(\theta)\,\Gamma(n+1)} +  \frac { \Gamma(n+\theta)}
       {\theta \,\Gamma(\theta)\,\Gamma(n)} \\
        &=  -\frac {\Gamma(n+\theta)}
        {\theta\,\Gamma(\theta)\, \Gamma(n)} \Big(\frac{n+\theta} {n} -1\Big)\\
        &=  -\frac {\Gamma(n+\theta)}
        {\Gamma(\theta)\, \Gamma(n+1)}\\
        &\sim  \frac {n^{\theta-1}}{\Gamma(\theta)}.
        \end{align*}
\end{proof}      
\begin{corollary}
\label{Cor:bounds}
For large enough positive constants $K_1, K_2$, 
and $K_3$, we have
$$  r_n^{(\theta)} \le K_1   n^{\theta-1},  
        \qquad r_n^{(\theta)}-r_{n-1}^{(\theta)} \le K_2 n^{\theta-2};$$
      $$ s_n^{(\theta)}-s_{n-1}^{(\theta)} = K_3n^{\theta-1},$$
      for all $n\ge 1$.
\end{corollary}
\subsection{Gaussian limit law}        
In this subsection, 
we obtain an asymptotic Gaussian law for $X_{n,1}^{(\theta)}$
by verifying the conditions of the martingale central limit theorem for 
$M_n^{(\theta)}$.
There are several sets of such conditions.
We use
conditional Lindeberg's condition and the conditional variance condition 
in~\cite{Hall}, pages 57--59. 

Conditional Lindeberg's condition 
requires that, for some positive sequence~$\xi_n^{(\theta)}$, 
and for any $\varepsilon>0$, we have
$$ U_n^{(\theta)} := \sum_{j = 1}^{n} 
\E \Big[ \Big( \frac{\nabla M_j^{(\theta)}}{\xi_n^{(\theta)}} \Big)^2 \indicator_
   {\big\{\big|\frac{\nabla M_j^{(\theta)}}{\xi_n^{(\theta)}} \big| > \varepsilon\big\}} \, \Big|  \, \field_{j - 1} \Big] \inprob  0,$$
and the conditional variance condition requires that, for some random variable 
$G_\theta \not \equiv 0$, we have
$$V_n^{(\theta)}:  = \sum_{j = 1}^{n} \E \Big[ \Big( \frac{\nabla M_j^{(\theta)}}{\xi_n^{(\theta)}} \Big)^2 \, \Big| \, \field_{j - 1} \Big] \inprob G_\theta.$$
When these conditions are satisfied, we get
$$\frac {M_n^{(\theta)}} {\xi_n^{(\theta)}} \ \convD 
\ \normal (0, G_\theta),$$
where the right-hand side is  a mixture of normally distributed random variates, with mixing variance $G_\theta$. In our case, we  
find out that $G_\theta$ is a constant, so the mixture has only
one normal random variate in it, which has a deterministic variance.
In the case of the number of vertices at
containment level 1, it turns out that 
$\xi_n^{(\theta)}$ is $n^{\theta - \frac 1 2}$.

The following uniform bound paves the way to the verification
of the two conditions of the martingale central limit theorem.  
\begin{lemma}
	\label{Lm:uniformbound}
	The absolute differences $|\nabla M_j^{(\theta)}|/ n^{\theta - 1}$ are  uniformly bounded in $j = 1, \ldots, n$.
\end{lemma}
\begin{proof}
	By the construction of the martingale, for each $ 1 \le j \le n $, we have
	\begin{align*}
	\big|\nabla M_j^{(\theta)}\big| &= \big|M_j^{(\theta)} - M_{j - 1}^{(\theta)}\big|\\ 
	&= \big|\big(r_j^{(\theta)} X_{j,1}^{(\theta)} + s_j^{(\theta)}\big) 
	   - \big(r_{j - 1}^{(\theta)}X_{j-1,1}^{(\theta)} + s_{j - 1}^{(\theta)}\big)\big| \\
&\le \big|r_j^{(\theta)} \big(X_{j-1,1}^{(\theta)}  - Q_{j,1}^{(\theta)} +1\big)
	   - r_{j - 1}X_{j-1,1}^{(\theta)}\big| + \big| s_j^{(\theta)} 
	     - s_{j - 1}^{(\theta)}\big|\\
	   &\le \big|r_j^{(\theta)} -r_{j-1}^{(\theta)}|\, \tau_{j-1}^{(\theta)}  + r_j ^{(\theta)} Q_{j,1}^{(\theta)} + r_j^{(\theta)}
 +  \big| s_j^{(\theta)}  - s_{j - 1}^{(\theta)}\big|.
	\end{align*}
	Recall that $Q_{n,1}^{(\theta)}$ is a hypergeometric random variable representing
	the number of vertices at containment level 1 in a sample
	of size $\theta -1$. Its maximum value is $\theta-1$. 
	
	From the bounds in~Corollary~\ref{Cor:bounds}, we obtain
	\begin{align*}
	\Big|\frac{\nabla M_j^{(\theta)}} {n^{\theta -1}}\Big| 
	   &\le  \frac 1 {n^{\theta - 1}} 
	      \big(K_2j^{\theta-2} \, \tau_{n-1}^{(\theta)}  
	         + (\theta-1) K_1 j^{\theta-1} 
                     + K_1 j^{\theta-1} + K_3n^{\theta -1} \big) \\
  &\le  \frac 1 {n^{\theta - 1}} 
	      \big(K_2n^{\theta-2} (n+\theta -1) 
         + \theta K_1 n^{\theta -1} \big) + K_3\\
          &= K_2  + \frac {(\theta-1)K_2}n+ \theta K_1 +K_3\\
          & \le  K_2  +(\theta-1)K_2+ \theta K_1 +K_3. 
	\end{align*}
\end{proof} 
\begin{lemma}
	\label{Lm:Lindeberg}
 For any $\varepsilon > 0$, we have
\begin{align*}
	U_n^{(\theta)} = \sum_{j = 1}^n \E \Big[ \Big( \frac{\nabla M_j^{(\theta)}}
	   {n^{\theta - \frac 1 2}} \Big)^2 \indicator_{\big\{\big|\frac{\nabla M_j^{(\theta)}}{n^{\theta-\frac 1 2 }} \big| > \varepsilon\big \}}\, \Big|\, \field_{j - 1} \Big] \inprob  0.
	\end{align*}
\end{lemma}
\begin{proof}
	By the uniform bound established in Lemma~\ref{Lm:uniformbound}, for every $\varepsilon>0$, 
	there exists a natural number $n_0 (\varepsilon) $, such that for all $ n \ge n_0 (\varepsilon)$, the sets 
	$\{|\nabla M_j^{(\theta)}  | > \varepsilon {n^{\theta-\frac 1 2}} \}$ are empty, which implies that the sequence $U_n$ converges almost surely to~$0$.
	This almost-sure convergence is stronger than the required in-probability convergence.
\end{proof}
\begin{lemma}
	\label{Lm:Variance}
	\begin{align*}
	V_n^{(\theta)} = \sum_{j = 1}^{n} \E \Big[ \Big( \frac{\nabla M_j^{(\theta)}}{n^{\theta -\frac 1 2}} \Big)^2 \, \Big | \, \field_{j - 1} \Big] \inprob  \frac{\theta-1}{\theta^2(2\theta-1)\, \Gamma^2(\theta)} .
	\end{align*}
\end{lemma}
\begin{proof}
	This is a rather lengthy calculation, but for the large part it goes in the same vein as the computations we encountered in the proof 
of Theorem~\ref{Theo:variance}. So, we only outline the salient features in the long 
chain of calculations.   
	
	Write 
	\begin{align*}
	V_n^{(\theta)} &= \frac{1}{n^{2\theta-1}}\sum_{j = 1}^{n} 
	          \E \big[\big(\nabla \big(r_j^{(\theta)}  X_{j,1}^{(\theta)}\big) 
	          + \nabla s_j^{(\theta)}\big)^2\big ]  \, \big|\, \field_{j - 1} \big]\\ 
	       &=  \frac{1}{n^{2\theta-1}} \sum_{j = 1}^{n} 
	       \E \big[\big(\nabla \big(r_j^{(\theta)} X_{j,1}^{(\theta)}\big)\big)^2 
	       + 2 \big(\nabla (r_j^{(\theta)} X_{j,1}^{(\theta)})\big) 
	       \nabla s_j^{(\theta)} + \big (\nabla  s_j^{(\theta)}\big)^2 \, 
	       \big|\, \field_{j - 1} \big]\\
	       &:= \frac{1}{n^{2\theta-1}}
	           \sum_{j = 1}^{n} \big(A_j^{(\theta)}
	               + B_j^{(\theta)} + D_j^{(\theta)}\big).
	\end{align*}

We take up each part separately, starting with	
		\begin{align*}
		A_j^{(\theta)} &:= \E \big[\big(\big(\nabla 
		     \big(r_j^{(\theta)} X_{j,1}^{(\theta)}\big)\big)^2  
		            \, \big|\, \field_{j - 1} \big] \\
		&= \E \big[\big(r_j^{(\theta)} X_{j,1}^{(\theta)} 
		   - r_{j-1}^{(\theta)} X_{j-1,1}^{(\theta)}\big)^2   \, \big|\, \field_{j - 1} \big]\\
		&= \big(r_j^{(\theta)}\big)^2\, \E\big[\big(X_{j,1}^{(\theta)}\big)^2 
		      \, |\,  \field_{j - 1}\big] + \big(r_{j - 1}^{(\theta)}\big)^2 
		      \big(X_{j-1,1}^{(\theta)}\big)^2\\
		       &\qquad\qquad  {} - 2 r_j^{(\theta)} r_{j - 1}^{(\theta)} X_{j-1,1}^{(\theta)} \, \E\big[X_{j,1}^{(\theta)} \, \big|\, \field_{j - 1}\big]\\
		      	&= \big(r_j^{(\theta)}\big)^2\, \E\big[\big(X_{j-1,1}^{(\theta)} - Q_{j,1}^{(\theta)} +1\big)^2 
		      \, \big | \, \field_{j - 1}\big] +\big( r_{j - 1}^{(\theta)}\big)^2 
		      \big(X_{j-1,1}^{(\theta)}\big)^2\\
		       & \qquad \qquad {} - 2 r_j^{(\theta)}
		            r_{j - 1} ^{(\theta)}  X_{j-1,1}^{(\theta)} \, 
		            \E\big[X_{j-1,1}^{(\theta)} 
               - Q_{j,1}^{(\theta)} +1 \, \big|\ \field_{j - 1}\big] ,
		\end{align*}
where we substituted the right-hand side 
of~(\ref{Eq:stoch1}) for $X_{j,1}^{(\theta)}$. 
Upon expanding, we get conditional expectations (given $\field_{j-1}$)
of both $(Q_{j,1}^{(\theta)})^2$ and $Q_{j,1}^{(\theta)}$. 
The variable $Q_{j,1}^{(\theta)}$ is 
Hypergeo$(\tau_{j-1}^{(\theta)}, X_{j-1,1}^{(\theta)}, \theta-1)$. 
The required
conditional expectations
are obtained from the hypergeometric distribution; see 
Section~\ref{Sec:notation}. 
 The second part is
\begin{align*}
B_j	^{(\theta)}   &= 
 \E \big[2 \big(\nabla (r_j^{(\theta)} X_{j,1}^{(\theta)})\big) 
	       \nabla s_j^{(\theta)}  \,   \big|\, \field_{j - 1} \big]\\
	       &= 2\, \E \big[ \big(r_j^{(\theta)} X_{j,1}^{(\theta)} - r_{j-1}^{(\theta)} X_{j-1,1}^{(\theta)})
	       \big(s_j^{(\theta)}-s_{j-1}^{(\theta)}\big)  \,  
	        \big|\, \field_{j - 1} \big]\\
	        &= 2\,  r_j^{(\theta)} \big(s_j^{(\theta)} - s_{j-1}^{(\theta)}\big)
	        \, \E\big[X_{j,1}^{(\theta)}\, |\, \field_{j-1}\big]
	       - 2 r_{j-1}^{(\theta)}  \big(s_j^{(\theta)}-s_{j-1}^{(\theta)}\big) X_{j-1,1}^{(\theta)}.
	\end{align*}
	
	The third part is
$$D_j^{(\theta)} =  
	       \E \big[\big (\nabla  s_j^{(\theta)}\big)^2 \, 
	       \big|\, \field_{j - 1} \big] = (s_j^{(\theta)}-s_{j-1}^{(\theta)})^2.$$
	
	We now put the three parts together and get 
	an expression for $V_n^{(\theta)}$ as a sum, in which 
	the summand is in terms of
 $r_j^{(\theta)}$ and $s_j^{(\theta)}$, and their backward differences, as well as $X_{j-1,1}^{(\theta)}$.
	 Toward simplified asymptotics, we use the asymptotic equivalents in Corollary~\ref{Lm:Lconcentration}
	for~$X_{j-1, 1}^{(\theta)}$, and for $r_j^{(\theta)}$, $s_j^{(\theta)}$,
	$\nabla r_j^{(\theta)}$, $\nabla s_j^{(\theta)}$, we use the asymptotics in 
	Lemma~\ref{Lm:rs}.

Huge cancellations take place, leaving
	\begin{align*}
	V_{n}^{(\theta)} &=  \frac 1 {n^{2\theta - 1}}
	     \Big( \sum_{j=1}^n   \frac{(\theta-1)^2}{\theta^2\, \Gamma^2(\theta)}  j^{2\theta-2} + \OL(j^{2\theta-3})\Big)\\ 
	 &=   \frac{(\theta-1)^2}{\theta^2(2\theta-1)\, \Gamma^2(\theta)} 
	          + \OL\Big(\frac 1 n\Big)\\
	          &\inL  \frac{(\theta-1)^2}{\theta^2(2\theta-1)\, \Gamma^2(\theta)}. 
	\end{align*}  
	This ${\cal L}_1$ convergence is stronger than the required 
	in-probability convergence.
\end{proof}

\begin{theorem}
\label{Theo:Smythe}
Let $X_{n,1}^{(\theta)}$ be the number of vertices in hyperrecursive tree
with hyperedges of size $\theta$ at age $n$. Then, as $n\to\infty$, 
we have
$$\frac {X_{n,1}^{(\theta)} - \frac n \theta} {\sqrt n}
\ \convD\ \normal \Big(0,  \frac{(\theta-1)^2}{\theta^2(2\theta-1)} \Big).$$
\end{theorem}
\begin{proof}
Having checked the conditions for the martingale central limit theorem,
we can ascertain that
\begin{align*}
\frac {M_n^{\theta}} {n^{\theta - \frac 1 2}} &=
\frac { \frac {\Gamma(n+\theta)}
        {\Gamma (\theta)\, \Gamma(n+1)} \, X_{n,1}^{(\theta)} - \frac {\Gamma(n+\theta+1) }{\theta\,\Gamma(\theta)\,\Gamma(n+1)}+1} {n^{\theta - \frac 1 2}}  \\
&= \frac { \frac {\Gamma(n+\theta)}
        {\Gamma (\theta)\, \Gamma(n+1)} \, X_{n,1}^{(\theta)} - \frac {\Gamma(n+\theta)} {\Gamma(\theta) \, \Gamma(n+1)}\big(\frac{n}{\theta}\big)- \frac {\Gamma(n+\theta) }{\Gamma(\theta)\,\Gamma(n+1))}+1} {n^{\theta - \frac 1 2}} \\
        & \convD\ \normal \Big(0,  \frac{(\theta-1)^2}{\theta^2(2\theta-1)\, \Gamma^2(\theta)} \Big).
\end{align*}
      We have $\big(- \frac {\Gamma(n+\theta) }{\,\Gamma(n+1)\Gamma(\theta)}+1\big) \big( n^{\frac 1 2 -\theta } \big)\to 0$, and so
 an application of Slutsky theorem~\cite{Slutsky} allows us to remove this term.
   By the Stirling approximation in~(\ref{Eq:Stirlingapprox}), we have 
   $$ \frac  {n^{\theta-1}\Gamma(n+1)} {\Gamma(n+\theta)}\to 1.$$
        An application of Slutsky theorem~\cite{Slutsky} yields       
$$
\frac { \frac 1
        {\Gamma (\theta)} \,( X_{n,1}^{(\theta)} - \frac n \theta)} {\sqrt n} 
   \convD\ N \Big(0,  \frac{(\theta-1)^2}
          {\theta^2(2\theta-1)\, \Gamma^2(\theta)} \Big),$$
          which is an equivalent statement to the one given in the theorem.
\end{proof}
\begin{remark}
In the very special case $\theta=2$,
the hyperrecursive tree is the standard uniform recursive tree.
In this case, $X_{n,1}^{(2)}$ is just a count of the leaves in the tree. 
Theorem~\ref{Theo:Smythe} recovers the 
result in~\cite{MahmoudSmythe012}
and generalizes it.
\end{remark}

\newpage
\section*{Appendix}
\begingroup
\allowdisplaybreaks
To 
find the variance of $X_{n,1}^{(\theta)}$, we need to solve first for $\E [\big(X_{n,1}^{(\theta)}\big)^2]$. We can do so through the following recurrence equation:
\begin{align*}
  \E \big[(X_{n,1}^{(\theta)} )^{2} \big] &= \E \big[\big( X_{n-1,1}^{(\theta)}-Q_{n,1}^{(\theta)}+1 \big)^2  \big] \\ 
&=\E \big[\big( X_{n-1,1}^{(\theta)}\big)^2+ \big(Q_{n,1}^{(\theta)}\big)^2 -2X_{n-1,1}^{(\theta)}Q_{n,1}^{(\theta)}-2Q_{n,1}^{(\theta)}+2X_{n-1,1}^{(\theta)}+1\big]\\
&=\E \big[\big( X_{n-1,1}^{(\theta)})^2 \big]+ \E \big[\big(Q_{n,1}^{(\theta)}\big)^2 \big] -2\,\E \big[X_{n-1,1}^{(\theta)}Q_{n,1}^{(\theta)} \big]\\
&\qquad{}-2\E \big[Q_{n,1}^{(\theta)} \big]+2\E \big[X_{n-1,1}^{(\theta)}\big]+1 \\
&=\E \big[\big( X_{n-1,1}^{(\theta)})^2 \big]+  \E \big[\E \big[\big(Q_{n,1}^{(\theta)}\big)^2 \given \field_{n-1}\big]\big] -2 \E \big[\E \big[X_{n-1,1}^{(\theta)}Q_{n,1}^{(\theta)}\given\field_{n-1}] \big]\\
& \qquad {} - 2 \,\E \big[\E \big[Q_{n,1}^{(\theta)} \given\field_{n-1}\big]\big]
     +2\, \E \big[X_{n-1,1}^{(\theta)}\big]+1 \nonumber\\
&= \E \big[\big( X_{n-1,1}^{(\theta)}\big)^2 \big] \nonumber\\
&\quad {}+\E \Big[\frac{X_{n-1,1}^{(\theta)}(\theta-1)(n+\theta-1-(\theta-1))(n+\theta-1-X_{n-1,1}^{(\theta)})}{(n+\theta-1)^2 (n+\theta-1-1)} \\
&\qquad\quad\quad {} +\Big(\frac{X_{n-1,1}(\theta -1)}{(n+\theta-1)}\Big)^2\Big] \\
&\qquad\quad\quad {} -\Big(\frac{2(\theta-1)}{n+\theta-1}\Big)\E \big[\big(X_{n-1,1}^{(\theta)}\big)^2 \big] 
-\Big(\frac{2(\theta-1)}{n+\theta-1}\Big)\, \E \big[X_{n-1,1}^{(\theta)} \big] \\
&\qquad\quad\quad {}+2\,\E \big[X_{n-1,1}^{(\theta)}\big]+1 \\
&=\frac{n(n-1)}{(n+\theta-1)(n+\theta-2)} \,  \E \big[\big( X_{n-1,1}^{(\theta)}\big)^2 \big] \\
& \qquad \qquad\qquad {}+ \Big(\frac{n(2n+3\theta-5)}{(n+\theta-1)(n+\theta-2)}\Big)\, \E \big[X_{n-1,1}^{(\theta)}\big]+1.
\end{align*}

Plugging in $\E [ X_{n-1,1}^{(\theta)}]=\frac{(n-1)}{\theta} + 1 +O(n^{-\theta+1})$ , we obtain 
\begin{align*}
 \E \big[(X_{n,1}^{(\theta)} )^{2} \big]&= \frac{n(n-1)}{(n+\theta-1)(n+\theta-2)}\, \, \E \big[\big( X_{n-1,1}^{(\theta)}\big)^2 \big]  \\
&\qquad {}+ \Big(\frac{n(2n+3\theta-5)}{(n+\theta-1)(n+\theta-2)}\Big)\Big(\frac{n-1}{\theta}+1+O\big(n^{1-\theta}\big) \Big)+1 \\
&=  \frac{n(n-1)}{(n+\theta-1)(n+\theta-2)}\, \E \big[\big( X_{n-1,1}^{(\theta)}\big)^2 \big] +\frac{2}{\theta}\,n+\frac{2\theta-1}{\theta}+O\Big(\frac 1 n \Big).
\end{align*}

Take $g_n=\frac{n(n-1)}{(n+\theta-1)(n+\theta-2)}$, and 
$h_n=\frac{2}{\theta}n+\frac{2\theta-1}{\theta}+O\big(\frac 1 n\big) $ 
in~(\ref{Eq:standard}). Under the initial condition 
$\E[(X_{0,1}^{(\theta)})^2]=\theta^2$, we can solve the  
preceding recurrence:
\begin{align*}
 \E \big[(X_{n,1}^{(\theta)} )^{2} \big]&= \theta^2\prod_{i=1}^{n}\frac{i(i-1)}{(i+\theta-1)(i+\theta-2)}\\
&\quad {}+  \sum_{i=1}^{n} \Big(\prod_{j=i+1}^{n} \frac{j(j-1)}{(j+\theta-1)(j+\theta-2)}\Big) \Big(\frac{2}{\theta}i+\frac{2\theta-1}{\theta}+O\Big(\frac{1}{i} \Big) \Big) \nonumber\\
&=\theta^2\Big(\frac{\Gamma(n+1)}{\Gamma(n+\theta) / \Gamma(\theta+1)} \Big)\Big(\frac{\Gamma(n)}{\Gamma(n+\theta-1) / \Gamma(\theta)} \Big) \nonumber\\
&\qquad {}+ \sum_{i=1}^{n} \Big(\frac{\Gamma(n+1)/\Gamma(i+1)}{\Gamma(n+\theta)/\Gamma(i+\theta)} \Big) \Big(\frac{\Gamma(n)/\Gamma(i)}{\Gamma(n+\theta-1)/\Gamma(i+\theta-1)} \Big) \nonumber\\
&\qquad \qquad\qquad {}\times \Big(\frac{2}{\theta}i+\frac{2\theta-1}{\theta}+O\Big(\frac{1}{i} \Big)  \Big)\nonumber\\
&=  O\big(n^{2-2\theta}\big)+ \Big(\frac{\Gamma(n+1)}{\Gamma(n+\theta)} \Big)^2 \Big(\frac{n+\theta-1}{n}\Big) \sum_{i=1}^{n} \Big(\frac{\Gamma(i+\theta)}{\Gamma(i+1)} \Big)^2 \nonumber\\
&\qquad {}\times \Big(\frac{i}{i+\theta-1}\Big) \Big(\frac{2}{\theta}i+\frac{2\theta-1}{\theta}+O\Big(\frac{1}{i} \Big) \Big)   \nonumber\\
&=  O\big(n^{2-2\theta}\big)+\Big(n^{1-\theta}+\frac{(1-\theta)(\theta)}{2} n^{-\theta} +  O\big(n^{-\theta-1}\big)\Big)^2  \Big(\frac{n+\theta-1}{n}\Big) \nonumber\\
&\qquad  {} \times  \sum_{i=1}^{n}\Big(i^{\theta-1}+\frac{\theta(\theta-1)}{2} i^{\theta-2} +  O\big(i^{\theta-3}\big)\Big)^2 \Big(\frac{i}{i+\theta-1}\Big) \nonumber\\
&\qquad\qquad\qquad {} \times \Big(\frac{2}{\theta}i+\frac{2\theta-1}{\theta}+O\Big(\frac{1}{i} \Big)\Big)  \nonumber\\
&=  O\big(n^{2-2\theta}\big)+\Big(n^{2-2\theta} + \theta(1-\theta)n^{1-2\theta} + O\big(n^{-2\theta}\big)\Big) \Big( \frac{n+\theta-1}{n}\Big) \nonumber\\
&\qquad {} \times  \sum_{i=1}^{n}\Big(\frac{2}{\theta}i^{2\theta-1}+\frac{2\theta^2-2\theta+1}{\theta} i^{2\theta-2} +  O\big(i^{2\theta-3}\big)\Big)\nonumber\\
&=  O\big(n^{2-2\theta}\big)+\Big(n^{2-2\theta} + \theta(1-\theta)n^{1-2\theta} + O\big(n^{-2\theta}\big)\Big) \Big( \frac{n+\theta-1}{n}\Big) \nonumber\\
&\qquad {} \times \Big( \frac{n^{2\theta}}{\theta^2}+ \frac{n^{2\theta-1}}{\theta} +\frac{2\theta^2-2\theta+1}{\theta(2\theta-1)} n^{2\theta-1} + O\big(n^{2\theta-2}\big) \Big) \nonumber\\
&=  O\big(n^{2-2\theta}\big)+\big(n^{2-2\theta} -(\theta-1)^2 n^{1-2\theta} + O\big(n^{-2\theta}\big)\big) \\
&\qquad {} \times \Big( \frac{n^{2\theta}}{\theta^2}+ \frac{2\theta}{2\theta-1}n^{2\theta-1} + O\big(n^{2\theta-2}\big) \Big) \nonumber\\
&= \frac{n^{2}}{\theta^2}+ \frac{5\theta^2-4\theta+1}{\theta^2(2\theta-1)}n+ O\big(1\big). \\
\end{align*}

We can now attain the variance as such:
\begin{align*}
\V \big[X_{n,1}^{(\theta)} \big] &=\E \big[\big(X_{n,1}^{(\theta)}\big)^2 \big] -\big(\E\big[X_{n,1}^{(\theta)} \big]\big)^2 \\
&=\frac{n^{2}}{\theta^2}+ \frac{5\theta^2-4\theta+1}{\theta^2(2\theta-1)}n+ O(1) - \Big( \frac{n}{\theta} +1+ O\big(n^{1-\theta} \big)\Big)^2\\
&=\frac{(\theta-1)^2}{\theta^2 (2\theta-1)}n +  O(1) \\
&\sim \frac{(\theta-1)^2}{\theta^2 (2\theta-1)}n .
\end{align*}

To solve for the covariance of $X_{n,1}^{(\theta)}$ and $X_{n,2}^{(\theta)}$, we need to first solve for  $\E\big[X_{n,1}^{(\theta)} X_{n,2}^{(\theta)}\big]$. We can do so through the following recurrence equation:
\begin{align*}
  \E \big[X_{n,1}^{(\theta)} X_{n,2}^{(\theta)} \big] 
&= \E \big[\big( X_{n-1,1}^{(\theta)}-Q_{n,1}^{(\theta)}+1 \big) \big( X_{n-1,2}^{(\theta)}-Q_{n,2}^{(\theta)}+Q_{n,1}^{(\theta)} \big) \big]  \\ 
&= \E \big[X_{n-1,1}^{(\theta)}X_{n-1,2}^{(\theta)}-Q_{n,1}^{(\theta)}X_{n-1,2}^{(\theta)} +X_{n-1,2}^{(\theta)}
-X_{n-1,1}^{(\theta)}Q_{n,2}^{(\theta)}  \\ 
&\qquad\qquad {} +Q_{n,1}^{(\theta)}Q_{n,2}^{(\theta)} -Q_{n,2}^{(\theta)} + X_{n-1,1}^{(\theta)} Q_{n,1}^{(\theta)}-\big(Q_{n,1}^{(\theta)}\big)^2 +Q_{n,1}^{(\theta)}  \big]\nonumber\\
&=  \E \big[ X_{n-1,1}^{(\theta)}X_{n-1,2}^{(\theta)}\big]- \E \big[Q_{n,1}^{(\theta)}X_{n-1,2}^{(\theta)}\big] +\E \big[X_{n-1,2}^{(\theta)}\big] \nonumber\\
&\qquad {} -\E \big[X_{n-1,1}^{(\theta)}Q_{n,2}^{(\theta)}\big] +\E \big[Q_{n,1}^{(\theta)}Q_{n,2}^{(\theta)}\big] -\E \big[Q_{n,2}^{(\theta)}\big] \nonumber\\
&\qquad {}+\E \big[X_{n-1,1}^{(\theta)} Q_{n,1}^{(\theta)}\big]-\E \big[\big(Q_{n,1}^{(\theta)}\big)^2\big] +\E \big[Q_{n,1}^{(\theta)}  \big] \nonumber \\ 
&=  \E \big[ X_{n-1,1}^{(\theta)}X_{n-1,2}^{(\theta)}\big]- \E \big[\E \big[Q_{n,1}^{(\theta)}X_{n-1,2}^{(\theta)}\given \field_{n-1} \big]\big] +\E \big[X_{n-1,2}^{(\theta)}\big]\big] \\
&\qquad {} -\E \big[\E \big[X_{n-1,1}^{(\theta)}Q_{n,2}^{(\theta)}\given \field_{n-1}\big]\big] +\E \big[\E \big[Q_{n,1}^{(\theta)}Q_{n,2}^{(\theta)}\given \field_{n-1}\big]\big] \\
& \qquad {} -\E \big[\E \big[Q_{n,2}^{(\theta)}\given \field_{n-1}\big]\big] +\E \big[\E \big[X_{n-1,1}^{(\theta)} 
Q_{n,1}^{(\theta)} \given \field_{n-1}\big]\big]\\
&\qquad {} -\E \big[\E \big[\big(Q_{n,1}^{(\theta)}\big)^2 \given \field_{n-1}\big]\big] +\E \big[\E \big[Q_{n,1}^{(\theta)} \given \field_{n-1} \big]\big] \\
&=  \E \big[ X_{n-1,1}^{(\theta)}X_{n-1,2}^{(\theta)}\big]- \frac{\theta-1}{n+\theta-1}\, \E \big[ X_{n-1,1}^{(\theta)}X_{n-1,2}^{(\theta)}\big] +\E \big[X_{n-1,2}^{(\theta)}\big] \\
&\qquad {}+\frac{\theta-1}{n+\theta-1}\, \E \big[ X_{n-1,1}^{(\theta)}X_{n-1,2}^{(\theta)}\big]  +\Big(\Big(\frac{\theta-1}{n+\theta-1}\Big)^2
\\&\qquad\qquad {} -\frac{(\theta-1)(n+\theta-1-(\theta-1)}{(n+\theta-1)^2 (n+\theta-1-1)}\Big)\E \big[ X_{n-1,1}^{(\theta)}
     X_{n-1,2}^{(\theta)}\big] \\
&\qquad {} - \frac{\theta-1}{n+\theta-1}\, 
    \E \big[X_{n-1,2}^{(\theta)}\big]+\frac{\theta-1}{n+\theta-1}\, \E \big[\big(X_{n-1,1}^{(\theta) }\big)^2\big]\\
&\qquad {}-\E \Big[\frac{X_{n-1}^{(\theta)}(\theta-1)(n+\theta-1-(\theta-1))(n+\theta-1-X_{n-1,1}^{(\theta)})}{(n+\theta-1)^2 (n+\theta-1-1)}  \\
&\qquad\qquad\qquad  {}+\Big(\frac{X_{n-1,1}(\theta -1)}{n+\theta-1}\Big)^2 \Big] +\frac{\theta-1}{n+\theta-1}\, \E \big[X_{n-1,1}^{(\theta) }\big] \\ 
&= \frac{n(n-1)}{(n+\theta-1)(n+\theta-2)}\, \E \big[ X_{n-1,1}^{(\theta)}X_{n-1,2}^{(\theta)}\big]\\
&\qquad {} +\frac{n}{n+\theta-1}\, \E \big[X_{n-1,2}^{(\theta)}\big]\\
&\qquad {}+ \frac{n(\theta-1)}{(n+\theta-1)(n+\theta-2)}\, \E \big[\big( X_{n-1,1}^{(\theta)}\big)^2\big] \\
&\qquad {}+ \frac{(\theta-1)(\theta-2)}{(n+\theta-1)(n+\theta-2)}\,  \E \big[ X_{n-1,1}^{(\theta)}\big]. \\
\end{align*}

Again, we will rely upon asymptotic equivalents of $\E \big[ X_{n,1}^{(\theta)}\big]$,  $\E \big[ X_{n,2}^{(\theta)}\big]$, and $\E \big[ \big( X_{n,1}^{(\theta)} \big)^2\big]$ to reduce the recursive equation to that of computable order. Plugging in the following relationships, we attain the following asymptotic relationship:
\begin{align*}
  \E \big[X_{n,1}^{(\theta)} X_{n,2}^{(\theta)} \big] &= \frac{n(n-1)}{(n+\theta-1)(n+\theta-2)}\,  \E \big[ X_{n-1,1}^{(\theta)}X_{n-1,2}^{(\theta)}\big] \\
&\qquad{} +\frac{n}{n+\theta-1}\,\Big(\frac{\theta-1}{\theta^2}(n-1+\theta) +O\Big(\frac{\ln(n)}{n^{\theta-1}} \Big) \Big) \nonumber\\
&\qquad{} + \frac{n(\theta-1)}{(n+\theta-1)(n+\theta-2)}\,  \Big(\frac{1}{\theta^2}(n-1)^2 \\
&\qquad{}+ \frac{5\theta^2-4\theta+1}{\theta^2(2\theta-1)}(n-1) \\
&\qquad{}+ \frac{10\theta^4-25\theta^3+29\theta^2-14\theta+2}{2\theta^2(2\theta-1)} +O\Big(\frac{1}{n}\Big)   \Big) \\
&\qquad{}+ \frac{(\theta-1)(\theta-2)}{(n+\theta-1)(n+\theta-2)}\,\Big(\frac{1}{\theta}(n-1) + 1 +O\big(n^{-\theta+1}\big)\Big) \\
&=\frac{n(n-1)}{(n+\theta-1)(n+\theta-2)}\,  \E \big[ X_{n-1,1}^{(\theta)}X_{n-1,2}^{(\theta)}\big] \\
&\qquad{}+\frac{2(\theta-1)}{\theta^2}n+ \frac{\theta-1}{2\theta-1}+O\Big(\frac{\ln(n)}{n}\Big).
\end{align*}

Take $g_n=\frac{n(n-1)}{(n+\theta-1)(n + \theta-2)}$, and
$h_n=\frac{2(\theta-1)}{\theta^2}n+ \frac{\theta-1}{2\theta-1}+O\big(\frac{\ln(n)}{n}\big)$ in~(\ref{Eq:standard}).
Noting that $\E \big[X_{0,1}^{(\theta)}X_{0,2}^{(\theta)}\big]=\theta(0)=0$,
 we can solve the preceding recurrence:
\begin{align*}
\E \big[ X_{n,1}^{(\theta)}X_{n,2}^{(\theta)}\big]&= 0\times
\Big(\prod_{i=1}^{n}\frac{i(i-1)}{(i+\theta-1)(i+\theta-2)}\Big) \\
&\qquad{} +  \sum_{i=1}^{n} \Big(\prod_{j=i+1}^{n} \frac{j(j-1)}{(j+\theta-1)(j+\theta-2)}\Big) \Big(\frac{2(\theta-1)}{\theta^2}i+ \frac{\theta-1}{2\theta-1} \\
&\qquad{} +O\Big(\frac{\ln(i)}{i}\Big) \Big) \\
&=\Big(\frac{\Gamma(n+1)}{\Gamma(n+\theta)} \Big)^2 \Big(\frac{n+\theta-1}{n}\Big)\sum_{i=1}^{n} \Big(\frac{\Gamma(i+\theta)}{\Gamma(i+1)} \Big)^2 \Big(\frac{i}{i+\theta-1} \Big) \\
&\qquad{} \times \Big(\frac{2(\theta-1)}{\theta^2}i+ \frac{\theta-1}{2\theta-1} +O\Big(\frac{\ln(i)}{i}\Big) \Big)\\
&=\big(n^{2-2\theta}+\theta(1-\theta)n^{1-2\theta}+ O\big(n^{-2\theta}\big) \big)\Big(\frac{n+\theta-1}{n}\Big) \\
& \qquad{} \times \sum_{i=1}^{n}\big( i^{2\theta-2}+(\theta^2-\theta)i^{2\theta-3} +O\big(i^{2\theta-4}\big)\big) \\
&\qquad{} \times \Big(\frac{2(\theta-1)}{\theta^2}i+ \frac{2-8\theta+9\theta^2-3\theta^3}{\theta^2(2\theta-1)} +O\Big(\frac{\ln(i)}{i}\Big) \Big)\\
&= \big(n^{2-2\theta}-(\theta-1)^2n^{1-2\theta}+ O\big(n^{-2\theta}\big) \big) \sum_{i=1}^{n}\Big(\frac{2(\theta-1)}{\theta^2}i^{2\theta-1} \\
&\qquad{} +\frac{4\theta^2-13\theta^3+17\theta^2-10\theta+2}{\theta^2(2\theta-1)}i^{2\theta-2} +O\big(i^{2\theta-3}\ln(i)\big)\Big) \\
&= \big(n^{2-2\theta}-(\theta-1)^2n^{1-2\theta}+ O\big(n^{-2\theta}\big) \big) \Big(\frac{(\theta-1)}{\theta^3}n^{2\theta} + \frac{\theta-1}{\theta^2}n^{2\theta-1}\nonumber\\
&\qquad{}+\frac{4\theta^2-13\theta^3+17\theta^2-10\theta+2}{\theta^2(2\theta-1)^2}n^{2\theta-1}+O\big(n^{2\theta-2} \ln(n)\big)\Big) \\
&= \big(n^{2-2\theta}-(\theta-1)^2n^{1-2\theta}+ O\big(n^{-2\theta}\big) \big)\\
& \quad{}\times \Big(\frac{\theta-1}{\theta^3}n^{2\theta} + \frac{4\theta^4-9\theta^3+9\theta^2-5\theta+1}{\theta^2 (2\theta-1)^2} n^{2\theta-1}+O\big(n^{2\theta-2} \ln(n)\big)\Big)\\
&= \Big(\frac{\theta-1}{\theta^3}\Big)n^2 +\frac{7\theta^4-16\theta^3+14\theta^2-6\theta+1}{\theta^3(2\theta-1)^2} n + O\big(\ln(n)\big).
\end{align*}

We can now attain the covariance:
\begin{align*}
\Cov \big[X_{n,1}^{(\theta)},X_{n,2}^{(\theta)} \big] &=\E \big[ X_{n,1}^{(\theta)}X_{n,2}^{(\theta)}\big] -\E\big[X_{n,1}^{(\theta)} \big]\E\big[X_{n,2}^{(\theta)} \big] \nonumber\\
&= \Big(\frac{\theta-1}{\theta^3}\Big)n^2 +\frac{7\theta^4-16\theta^3+14\theta^2-6\theta+1}{\theta^3(2\theta-1)^2}n + O\big(\ln(n)\big)\nonumber \\
&- \Big(\frac{n}{\theta}+1+O\big(n^{1-\theta}\big)\Big)\Big(\Big(\frac{\theta-1}{\theta^2}\Big)(n+\theta) +O\Big(\frac{\ln(n)}{n} \Big) \Big) \nonumber\\
&= -\frac{\theta^4-4\theta^2+4\theta-1}{\theta^3(2\theta-1)^2}n +O\big(\ln(n)\big) \\
&\sim  - \frac{(\theta-1)^2(\theta^2+2\theta-1)}
{\theta^3(2\theta-1)^2}\,n. 
\end{align*}

Finally, to solve for the variance of $X_{n,2}^{(\theta)}$, we need to solve for $\E [\big(X_{n,2}^{(\theta)} \big)^{2}]$. We can do so through the following recurrence equation:

\begin{align*}
\E \big[\big(X_{n,2}^{(\theta)}\big)^2 \big] &= \E \big[\big(X_{n-1,2}^{(\theta)}-Q_{n,2}^{(\theta)}+Q_{n,1}^{(\theta)}\big)^2  \big]  \\ 
&= \E \big[ \big(X_{n-1,2}^{(\theta)}\big)^2 + \big(Q_{n,2}^{(\theta)} \big)^2+\big(Q_{n,1}^{(\theta)}\big)^2 +2X_{n-1,2}^{(\theta)}Q_{n,1}^{(\theta)} \\
&  \qquad{} -2X_{n-1,2}^{(\theta)}Q_{n,2}^{(\theta)}   -2Q_{n,2}^{(\theta)}Q_{n,1}^{(\theta)}\big] \\
&= \E \big[ \big(X_{n-1,2}^{(\theta)}\big)^2 \big] +\E \big[\E \big[ \big(Q_{n,2}^{(\theta)} \big)^2 \given \field_{n-1}\big] \big]+\E \big[\E \big[\big(Q_{n,1}^{(\theta)}\big)^2 \given \field_{n-1}\big] \big] \\
&\qquad{}+\E \big[\E \big[2X_{n-1,2}^{(\theta)}Q_{n,1}^{(\theta)} \given\field_{n-1}\big]\big]-\E \big[\E \big[2X_{n-1,2}^{(\theta)}Q_{n,2}^{(\theta)} \given \field_{n-1}\big]\big]\\
&-\E \big[\E \big[2Q_{n,2}^{(\theta)}Q_{n,1}^{(\theta)} \given \field_{n-1}\big]\big] \nonumber\\
&= \E \big[ \big(X_{n-1,2}^{(\theta)}\big)^2 \big] \nonumber \\
&\quad{} +\E \Big[ \frac{X_{n-1,1}^{(\theta)}(\theta-1)(n)(n+\theta-1-X_{n-1,1}^{(\theta)})}{(n+\theta-1)^2 (n+\theta-2)} +\frac{\big(X_{n-1,1}^{(\theta)}\big)^2 (\theta -1)^2}{(n+\theta-1)^2} \Big] \nonumber\\
&\quad{}+\E \Big[ \frac{X_{n-1,2}^{(\theta)}(\theta-1)(n)(n+\theta-1-X_{n-1,2}^{(\theta)})}{(n+\theta-1)^2 (n+\theta-2)} +\frac{\big(X_{n-1,2}^{(\theta)}\big)^2 (\theta -1)^2}{(n+\theta-1)^2} \Big] \nonumber\\
&\quad{}+\frac{2(\theta-1)}{n+\theta-1}\, \E \big[X_{n-1,1}^{(\theta)}X_{n-1,2}^{(\theta)} \big] - \frac{2(\theta-1)}{n+\theta-1}\, \E \big[\big(X_{n-1,2}^{(\theta)}\big)^2 \big]\nonumber\\
&\quad{} -2\Big(\Big(\frac{\theta-1}{n+\theta-1}\Big)^2-\frac{(\theta-1)
n}
{(n+\theta-1)^2(n+\theta-2)}\Big)\E \big[X_{n-1,1}^{(\theta)}X_{n-1,2}^{(\theta)}  \big] \nonumber\\
&= \frac{n(n-1)}{(n+\theta-1)(n+\theta-2)}\, \E \big[ \big(X_{n-1,2}^{(\theta)}\big)^2 \big] \\
& \quad{}+\frac{(\theta-1)(\theta-2)}{(n+\theta-1)(n+\theta-2)}\, \E \big[ \big(X_{n-1,1}^{(\theta)}\big)^2 \big] \\
&\quad{}+\frac{2n(\theta-1)}{(n+\theta-1)(n+\theta-2)}\, \E \big[ X_{n-1,1}^{(\theta)}X_{n-1,2}^{(\theta)}\big]\\
&\quad{}+\frac{n(\theta-1)}{(n+\theta-1)(n+\theta-2)}\, \big(\E \big[ X_{n-1,1}^{(\theta)}\big]+\E \big[X_{n-1,2}^{(\theta)}\big]   \big). 
\end{align*}

Plugging in the following asymptotic relationships for $\E[ X_{n,1}^{(\theta)}]$, $\E[ X_{n,2}^{(\theta)}]$, $\E[ \big( X_{n,1}^{(\theta)} \big)^2]$, 
and $\E[ X_{n,1}^{(\theta)}X_{n,2}^{(\theta)}]$, we attain the following asymptotic recurrence:
\begin{align*}
\E \big[\big(X_{n,2}^{(\theta)}\big)^2 \big] &= \frac{n(n-1)}{(n+\theta-1)(n+\theta-2)}\,\E \big[ \big(X_{n-1,2}^{(\theta)}\big)^2 \big] \\
& \qquad{}+\frac{(\theta-1)(\theta-2)}{(n+\theta-1)(n+\theta-2)}\\ & \qquad \qquad {} \times\Big(\frac{1}{\theta^2}(n-1)^2 + \frac{5\theta^2-4\theta+1}{\theta^2(2\theta-1)}(n-1) +O\big(1 \big) \Big) \\
&\qquad{}+\Big(\frac{2n(\theta-1)}{(n+\theta-1)(n+\theta-2)}\Big)\Big(\frac{\theta-1}{\theta^3}(n-1)^2  \\
&\qquad\qquad \qquad{}+\frac{1-6\theta+14\theta^2-16\theta^3+7\theta^4}{(1-2\theta)^2\theta^3}\, (n-1)+ O\big(\ln(n)\big)\Big)\\
&\qquad{}+\frac{(\theta-1)n}{(n+\theta-1)(n+\theta-2)}\,\Big(\frac{1}{\theta}(n-1) + 1 +O\big(n^{-\theta+1}\big) \\
&\qquad\qquad\qquad{}+\Big(\frac{\theta-1}{\theta^2}(n-1+\theta) +O\Big(\frac{\ln(n)}{n^{\theta-1}} \Big)\Big)   \Big) \\
&=\Big(\frac{n(n-1)}{(n+\theta-1)(n+\theta-2)}\Big)\E \big[ \big(X_{n-1,2}^{(\theta)}\big)^2 \big] \\
& \qquad{}+\frac{2(\theta-1)^2}{\theta^3}n + \frac{(10\theta^3-16\theta^2+7\theta-1)(\theta-1)}{\theta^2(2\theta-1)^2}+O\Big(\frac{\ln(n)}{n} \Big)\\
\end{align*}

Take 
$g_n=\frac{n(n-1)}{(n+\theta-1)(n + \theta-2)}$, and
$h_n=\frac{2(\theta-1)^2}{\theta^3}n + \frac{(10\theta^3-16\theta^2+7\theta-1)(\theta-1)}{\theta^2(2\theta-1)^2}+O\big(\frac{\ln(n)}{n} \big)$ in~(\ref{Eq:standard}).
Noting that $\E[\big(X_{0,2}^{(\theta)}\big)^2]=0$,
we can solve the preceding recurrence:
\begin{align*}
\E \big[\big(X_{n,2}^{(\theta)}\big)^2 \big] &=  0
\times\prod_{i=1}^{n}\frac{i(i-1)}{(i+\theta-1)(i+\theta-2)} \\
&\qquad{} +  \sum_{i=1}^{n} \Big(\prod_{j=i+1}^{n} \frac{j(j-1)}{(j+\theta-1)(j+\theta-2)}\Big) \Big(\frac{2(\theta-1)^2}{\theta^3}\, i \\
&\qquad\qquad\qquad+ \frac{(10\theta^3-16\theta^2+7\theta-1)(\theta-1)}{\theta^2(2\theta-1)^2}+O\Big(\frac{\ln(i)}{i} \Big) \Big) \\
&= \Big(\frac{\Gamma(n+1)}{\Gamma(n+\theta)} \Big)^2 \Big(\frac{n+\theta-1}{n}\Big)\sum_{i=1}^n \Big(\frac{\Gamma(i+\theta)}{\Gamma(i+1)} \Big)^2 \Big(\frac{i}{i+\theta-1} \Big) \nonumber\\
&\qquad{} \times  \Big(\frac{2(\theta-1)^2}{\theta^3}i+ \frac{(10\theta^3-16\theta^2+7\theta-1)(\theta-1)}{\theta^2(2\theta-1)^2}+O\Big(\frac{\ln(i)}{i} \Big) \Big) \nonumber\\
&= \big(n^{2-2\theta}-(\theta-1)^2 n^{1-2\theta}+ O\big(n^{-2\theta}\big) \big) \\
& \qquad{} \times \sum_{i=1}^{n}\big( i^{2\theta-2}+(\theta^2-\theta)i^{2\theta-3} +O\big(i^{2\theta-4}\big)\big) \nonumber\\
&\qquad{} \times  \Big(\frac{2(\theta-1)^2}{\theta^3}i+ \frac{2\theta^5+6\theta^4-27\theta^3+30\theta^2-13\theta+2}{\theta^3(2\theta-1)^2}+O\Big(\frac{\ln(i)}{i} \Big) \Big) \nonumber\\
&= \big(n^{2-2\theta}-(\theta-1)^2 n^{1-2\theta}+ O\big(n^{-2\theta}\big) \big) \\
&  \qquad{} \times\frac{(\theta-1)^2}{\theta^3} \sum_{i=1}^{n}\Big(2i^{2\theta-1}+\frac{8\theta^4-14\theta^3+20\theta^2-11\theta+2}{(2\theta-1)^2}\, i^{2\theta-2} \\
&\qquad\qquad\qquad\qquad\qquad+O\Big(\frac{\ln(i)}{i} \Big) \Big) \\
&= \big(n^{2-2\theta}-(\theta-1)^2 n^{1-2\theta}+ O\big(n^{-2\theta}\big) \big)\\
& \qquad{} \times  \frac{(\theta-1)^2}{\theta^3} \Big(\frac{n^{2\theta}}{\theta} + n^{2\theta-1}+\frac{8\theta^4-14\theta^3+20\theta^2-11\theta+2}{(2\theta-1)^3} n^{2\theta-1} \\
& \qquad\qquad{}+O\big(n^{2\theta-3}\ln(n) \big)\Big) \\
&=  \frac{(\theta-1)^2}{\theta^3} \big(n^{2-2\theta}- (\theta-1)^2 n^{1-2\theta} + O\big(n^{-2\theta}\big)\big)  \\
& \qquad {}\times \Big( \frac{1}{\theta}\, n^{2\theta} +\frac{8\theta^4 -6\theta^3+8\theta^2-5\theta+1}{(2\theta-1)^3}\, n^{2\theta-1} +O\big(n^{2\theta-2}\ln(n) \big)\Big) \\
&= \frac{(\theta-1)^2}{\theta^4} n^2 + \frac{(22\theta^4-30\theta^3+20\theta^2-7\theta+1) (\theta-1)^2}{\theta^4(2\theta-1)^3}n+ O\big(\ln(n)\big).
\end{align*}

We can now attain the variance as such:
\begin{align*}
\V \big[X_{n,2}^{(\theta)} \big] &=\E \big[\big(X_{n,2}^{(\theta)}\big)^2 \big] -\big(\E\big[X_{n,2}^{(\theta)} \big]\big)^2\\
&= \frac{(\theta-1)^2}{\theta^4}\, n^2 + \frac{(22\theta^4-30\theta^3+20\theta^2-7\theta+1) (\theta-1)^2}{\theta^4(2\theta-1)^3}\, n\\   
&\qquad {}+ O(\ln (n)) - \Big(\frac{\theta-1}{\theta^2}(n+\theta) +O\Big(\frac{\ln (n)} {n^{\theta-1}} \Big) \Big)^2 \\
&= \Big(\frac{(22\theta^4-30\theta^3+20\theta^2-7\theta+1) (\theta-1)^2}{\theta^4(2\theta-1)^3}-\frac{2(\theta-1)^2}{\theta^3}\Big)\, n \\
&\qquad {} +O(\ln (n)) \\
&= \frac{(\theta-1)^2 (6\theta^4-6\theta^3+8\theta^2-5\theta+1)}{\theta^4(2\theta-1)^3}\, n+O(\ln (n))\\
&\sim \frac{(\theta-1)^2 (6\theta^4-6\theta^3+8\theta^2-5\theta+1)}{\theta^4(2\theta-1)^3}\, n.
\end{align*}
\endgroup

\begin{thebibliography}{99}
\bibitem {Barton} David, F.\ and Barton, E.\ (1962).     
         {\em Combinatorial Chance}. 
         Charles Griffin, London.
\bibitem{Drmota}
         Drmota, M.\ (2009).
        {\em Random Trees: An Interplay Between Combinatorics and Probability},    
         Springer, New York.
\bibitem{Graham}
         Graham, R., Knuth, D.\ and Patashnik, O.\ (1994).
            {\em Concrete Mathematics}.
         Addison-Wesley, Reading, Massachusetts.
\bibitem {Hall}
Hall, P. and Heyde, C.\ (1980).
{\em Martingale Limit Theory and Its Application}.
Academic Press, Inc., New York.
\bibitem{Hofri}
         Hofri, M.\ and Mahmoud, H.\ (2018).
         {\em Algorithmics of Nonuniformity: Tools and Paradigms}, 
         CRC Press, Boca Raton, Florida.
\bibitem {QQ}       
        Feng, Q., Mahmoud, H. and Panholzer, A.\ (2008). 
        Phase changes in subtree varieties in random recursive trees 
        and binary search trees. 
        {\it SIAM Journal on Discrete Mathematics} {\bf 22}, 160--184.
\bibitem{Karonski}
   Frieze, A.\ and Karo\' nski (2015).
   {\em Introduction to Random Graphs}.
   Cambridge University Press.
\bibitem{Kendall}  
         Kendall, M., Stuart , A.\ and Ord, K.\ ( 1987). 
         {\em Advanced Theory of Statistics, Vol. I: Distribution Theory}. 
         Oxford University Press.
\bibitem{Mah2019}
Mahmoud, H.\ (2019). 
         Local and global degree profiles of randomly grown self-similar 
         hooking networks under uniform and preferential attachment, 
          {\em Advances in Applied Mathematics} {\bf 111}, 101930.
\bibitem{MahmoudSmythe012}
            Mahmoud, H. and Smythe, R. (1992). 
            Asymptotic joint normality of outdegrees of nodes in 
            random recursive trees. 
            {\em Random Structures and Algorithms} {\bf 3}, 255--266.
\bibitem{Slutsky}
           Slutsky, E.\ (1925). Uber stochastische asymptoten und grenzwerte.
           Metron {\bf5
           }, 3--89.
\bibitem {Smythe} 
Smythe, R. and Mahmoud, H.\ (1996).
A survey of recursive trees.
{\em Theory of Probability and Mathematical Statistics}
{\bf 51}, 1--29 (appeared in Ukrainian in (1994)).
\bibitem{Tricomi}
       Tricomi, F. 
       and Erd\'elyi, A.\ (1951). 
       The asymptotic expansion of a ratio of gamma functions.
       {\em Pacific Journal of Mathematics} {\bf 1}, 133--142.
\bibitem{Williams}
       Williams, D.\ (1991). 
       {\em Probability with Martingales}. Cambridge University Press,     
       New York.
\end{thebibliography}
\end{document}